\documentclass[graybox]{svmult}

\usepackage{mathptmx}       
\usepackage{helvet}         
\usepackage{courier}        
\usepackage{type1cm}        
\usepackage{makeidx}         
\usepackage{graphicx}        
\usepackage{multicol}        
\usepackage[bottom]{footmisc}

\makeindex             

\usepackage[utf8]{inputenc}
\usepackage[T1]{fontenc}
\usepackage{amsmath}
\usepackage{amssymb}
\usepackage[english]{babel}
\usepackage{longtable}

\usepackage{cite}

\usepackage[inline]{enumitem}  
\setlist{labelindent=1pt,itemsep=.5em}
\setlist[itemize]{leftmargin=1cm}
\setlist[enumerate]{itemindent=0em,leftmargin=1cm}

\usepackage[draft=false,hidelinks,colorlinks=true,linkcolor=blue,citecolor=black,urlcolor=black,anchorcolor=blue,bookmarks=false,hypertexnames=true,backref=page]{hyperref}

\allowdisplaybreaks[3]
\smartqed

\newcommand{\bracket}[1]{\left[#1\right]}
\newcommand{\para}[1]{\left(#1\right)}

\newcommand{\Id}{\operatorname{Id}}

\begin{document}
\smartqed

\title*{On properties and classification of a class of $4$-dimensional $3$-Hom-Lie algebras with a nilpotent twisting map}
\titlerunning{properties and classification of a class of $4$-dimensional $3$-Hom-Lie algebras}
\author{Abdennour Kitouni \and Sergei Silvestrov}
\authorrunning{A. Kitouni, S. Silvestrov}

\institute{
Abdennour Kitouni \at Division of Mathematics and Physics,
School of Education, Culture and Communication,
 M\"{a}lardalen University, Box 883, 72123 V{\"a}ster{\aa}s, Sweden. \\
 \email{abdennour.kitouni@mdu.se}
\and
Sergei Silvestrov (corresponding author) \at Division of Mathematics and Physics,
School of Education, Culture and Communication,
 M\"{a}lardalen University, Box 883, 72123 V{\"a}ster{\aa}s, Sweden. \\
 \email{sergei.silvestrov@mdu.se}
}

\date{}
\maketitle

\abstract*{}

\abstract{The aim of this work is to investigate the properties and classification of an interesting class of $4$-dimensional $3$-Hom-Lie algebras with a nilpotent twisting map $\alpha$ and eight structure constants as parameters. Derived series and central descending series are studied for all algebras in this class and are used to divide it into five non-isomorphic subclasses. The levels of solvability and nilpotency of the $3$-Hom-Lie algebras in these five classes are obtained. Building up on that, all algebras of this class are classified up to Hom-algebra isomorphism. Necessary and sufficient conditions for multiplicativity of general $(n+1)$-dimensional $n$-Hom-Lie algebras as well as for algebras in the considered class are obtained in terms of the structure constants and the twisting map. Furthermore, for some algebras in this class, it has been determined whether the terms of the derived and central descending series are weak subalgebras, Hom-subalgebras, weak ideals or Hom-ideals.
}

\section{Introduction}
\label{sec:intro}
Hom-Lie algebras and more general quasi-Hom-Lie algebras where introduced first by Hartwig, Larsson and Silvestrov in \cite{HLS}, where the general quasi-deformations and discretizations of Lie algebras of vector fields using more general $\sigma$-derivations (twisted derivations) and a general method for construction of deformations of Witt and Virasoro type algebras based on twisted derivations have been developed,
initially motivated by the $q$-deformed Jacobi identities observed for the $q$-deformed algebras in physics, $q$-deformed versions of homological algebra and discrete modifications of differential calculi
\cite{AizawaSaito,ChaiElinPop,ChaiIsLukPopPresn,ChaiKuLuk,ChaiPopPres,CurtrZachos1,DamKu,DaskaloyannisGendefVir,Hu,Kassel92,LiuKQuantumCentExt,LiuKQCharQuantWittAlg,LiuKQPhDthesis}.
The general abstract quasi-Lie algebras and the subclasses of quasi-Hom-Lie algebras and Hom-Lie algebras as well as their general colored (graded) counterparts have been introduced in \cite{HLS,LarssonSilvJA2005:QuasiHomLieCentExt2cocyid,LSGradedquasiLiealg,Czech:witt,LS2}.
Subsequently, various classes of Hom-Lie admissible algebras have been considered in \cite{ms:homstructure}. In particular, in \cite{ms:homstructure}, the Hom-associative algebras have been introduced and shown to be Hom-Lie admissible, that is leading to Hom-Lie algebras using commutator map as new product, and in this sense constituting a natural generalization of associative algebras, as Lie admissible algebras leading to Lie algebras via commutator map as new product.
In \cite{ms:homstructure}, moreover, several other interesting classes of Hom-Lie admissible algebras generalizing some classes of non-associative algebras, as well as examples of finite-dimensional Hom-Lie algebras have been described.
Hom-algebras structures are very useful since Hom-algebra structures of a given type include their classical counterparts and open more possibilities for deformations, extensions of cohomological structures and representations. Since these pioneering works \cite{HLS,LarssonSilvJA2005:QuasiHomLieCentExt2cocyid,LS2,LSGradedquasiLiealg,LS3,ms:homstructure}, Hom-algebra structures have developed in a popular broad area with increasing number of publications in various directions (see for example
\cite{AmmarEjbehiMakhlouf:homdeformation,BenMakh:Hombiliform,ElchingerLundMakhSilv:BracktausigmaderivWittVir,LarssonSilvJA2005:QuasiHomLieCentExt2cocyid,LarssonSigSilvJGLTA2008,LarssonSilvestrovGLTMPBSpr2009:GenNComplTwistDer,MakhSil:HomHopf,MakhSilv:HomAlgHomCoalg,MakhSilv:HomDeform,MakSil:HomLieAdmissibleHomCoalgHomHopf,RichardSilvestrovJA2008,shenghomrep,SigSilvGLTbdSpringer2009,YauHomEnv,YauHomHom}
and references therein).

Ternary Lie algebras appeared first in generalization of Hamiltonian mechanics by Nambu  \cite{Nambu:GenHD}. Besides Nambu mechanics, $n$-Lie algebras revealed to have many applications in physics. The mathematical algebraic foundations of Nambu mechanics have been developed by Takhtajan in
\cite{Takhtajan:foundgenNambuMech}. Filippov, in \cite{Filippov:nLie} independently introduced and studied structure of $n$-Lie algebras and Kasymov \cite{Kasymov:nLie} investigated their properties. Properties of $n$-ary algebras, including solvability and nilpotency, were studied in \cite{Carlsson:n-aryalgebras,Kasymov:nLie,VainermanKerner:n-aryalgebras}.
Kasymov \cite{Kasymov:nLie} pointed out that $n$-ary multiplication allows for several different definitions of solvability and nilpotency in $n$-Lie algebras, and studied their properties. Further properties, classification, and connections of $n$-ary algebras to other structures such as bialgebras, Yang-Baxter equation and Manin triples for $3$-Lie algebras were studied in
\cite{BaiRBaiCWang:realizations3Liealg,BaiGuoSheng:BialgYangBaxtereqManintri3Liealg,BaiWuLiZhou:Constrnplus1nLiealg,BaiRSongZhang:nLieclas,BaiRWangXiaoAn:nLieclaschar2,BaiRAnLi:CentStrnLiealg,BaiRChenmeng:FrattinisubalgnLiealg,BaiRMengD:CentrextnLiealg,BaiRMengD:CentrnLiealg,BaiRZhangLiShi:Innerderivalgnplus1dimnLiealg,Kasymov:nLie}.
The structure of $3$-Lie algebras induced by Lie algebras, classification of $3$-Lie algebras and application to constructions of B.R.S. algebras have been considered in \cite{Abramov2018:WeilAlg3LiealgBRSalg,AbramovLatt2016:classifLowdim3Liesuperalg,Abramov2017:Super3LiealgebrasinducedsuperLiealg}.
Interesting constructions of ternary Lie superalgebras in connection to superspace extension of Nambu-Hamilton equation is considered in \cite{AbramovLatt2020TernaryLieSuper}.
In \cite{CasasLodayPirashvili:Leibniznalg}, Leibniz $n$-algebras have been studied.
The general cohomology theory for $n$-Lie algebras and Leibniz $n$-algebras was established in \cite{DalTakh,Rotkiewicz:cohomringnLiealg,Takhtajan:cohomology}.
The structure and classification of finite-dimensional $n$-Lie algebras were considered in \cite{BaiRSongZhang:nLieclas,Filippov:nLie,LingW:StrnLiealgPhThes1993} and many
other authors. For more details of the theory and applications of $n$-Lie algebras, see \cite{DeAzcarragaIzquierdo:nAryalgrevappl}
and references therein.

Classifications of $n$-ary or Hom generalizations of Lie algebras have been considered, either in very special cases or in low dimensions. The classification of $n$-Lie algebras of dimension up to $n+1$ over a field of characteristic $p\neq 2$ has been completed by Filippov \cite{Filippov:nLie} using the specific properties of $(n+1)$-dimensional $n$-Lie algebras that make it possible to represent their bracket by a square matrix in a similar way as bilinear forms, the number of cases obtained depends on the properties of the base field, the list is ordered by ascending dimension of the derived ideal, and among them, one nilpotent algebra, and a class of simple algebras which are all isomorphic in the case of an algebraically closed field, the remaining algebras are $k$-solvable for some $2\leq k \leq n$ depending on the algebra. These simple algebras are proven to be the only simple finite-dimensional $n$-Lie algebras in \cite{LingW:StrnLiealgPhThes1993}. The classification of $(n+1)$-dimensional $n$-Lie algebras over a field of characteristic $2$ has been done by Bai, Wang, Xiao, and An \cite{BaiRWangXiaoAn:nLieclaschar2} by finding and using a similar result in characteristic $2$. Bai, Song and Zhang \cite{BaiRSongZhang:nLieclas} classify the $(n+2)$-dimensional $n$-Lie algebras over an algebraically closed field of characteristic 0 using the fact that an $(n+2)$-dimensional $n$-Lie algebra has a subalgebra of codimension $1$ if the dimension of its derived ideal is not $3$, thus constructing most of the cases as extensions of the $(n+1)$-dimensional $n$-Lie algebras listed by Filippov. In \cite{SimpleLinearlyCompact}, Cantarini and Kac classified all simple linearly compact $n$-Lie superalgebras, which turned out to be $n$-Lie algebras, by finding a bijective correspondence between said algebras and a special class of transitive $\mathbb{Z}$-graded Lie superalgebras, the list they obtained consists of four representatives,
one of them is the $(n+1)$-dimensional vector product $n$-Lie algebra, and the remaining three are infinite-dimensional $n$-Lie algebras.

Classifications of $n$-Lie algebras in higher dimensions have only been studied in particular cases. Metric $n$-Lie algebras, that is $n$-Lie algebras equipped with a non-degenerate compatible bilinear form, have been considered and classified, first in dimension $n+2$ by Ren, Chen and Liang \cite{metric_n+2} and dimension $n+3$ by Geng, Ren and Chen \cite{metric_n+3}, and then in dimensions $n+k$ for $2\leq k\leq n+1$ by Bai, Wu and Chen \cite{metric_n+k}. The classification is based on the study of the Levi decomposition, the center and the isotropic ideals and properties around them. Another case that has been studied is the case of nilpotent $n$-Lie algebras, more specifically nilpotent $n$-Lie algebras of class $2$. Eshrati, Saeedi and Darabi  \cite{lowdimnil-nLie} classify $(n+3)$-dimensional nilpotent $n$-Lie algebras and $(n+4)$-dimensional nilpotent $n$-Lie algebras of class $2$ using properties introduced in \cite{char_nil_nLie,multiplier_nLie}. Similarly  Hoseini, Saeedi and Darabi \cite{2s-nil-dim:n+5} classify $(n+5)$-dimensional nilpotent $n$-Lie algebras of class $2$. In \cite{2s-nil-dim:n+6}, Jamshidi,  Saeedi and Darabi  classify $(n+6)$-dimensional nilpotent $n$-Lie algebras of class $2$ using the fact that such algebras factored by the span of a central element give $(n+5)$-dimensional nilpotent $n$-Lie algebras of class $2$, which were classified before. There has been a study of the classification of $3$-dimensional $3$-Hom-Lie algebras with diagonal twisting maps by Ataguema, Makhlouf and Silvestrov in \cite{AtMaSi:GenNambuAlg}.

Hom-type generalization of $n$-ary algebras, such as $n$-Hom-Lie algebras and other $n$-ary Hom algebras of Lie type and associative type, were introduced in \cite{AtMaSi:GenNambuAlg}, by twisting the defining identities by a set of linear maps. The particular case, where all these maps are equal and are algebra morphisms has been considered and a way to generate examples of $n$-ary Hom-algebras from $n$-ary algebras of the same type have been described.
Further properties, construction methods, examples, representations, cohomology and central extensions of $n$-ary Hom-algebras have been considered in
\cite{AmmarMabMakh:naryhomrep,ams:ternary,ams:n,kms:nhominduced,YauGenCom,YauHomNambuLie}.
These generalizations include $n$-ary Hom-algebra structures generalizing the $n$-ary algebras of Lie type including $n$-ary Nambu algebras, $n$-ary Nambu-Lie algebras and $n$-ary Lie algebras, and $n$-ary algebras of associative type including $n$-ary totally associative and $n$-ary partially associative algebras.
In \cite{kms:narygenBiHomLieBiHomassalgebras2020}, constructions of $n$-ary generalizations of BiHom-Lie algebras and BiHom-associative algebras have been considered. Generalized derivations of $n$-BiHom-Lie algebras have been studied in \cite{BenAbdeljElhamdKaygorMakhl201920GenDernBiHomLiealg}. Generalized derivations of multiplicative $n$-ary Hom-$\Omega$ color algebras have been studied in \cite{BeitesKaygorodovPopov}.
Cohomology of Hom-Leibniz and $n$-ary Hom-Nambu-Lie superalgebras has been considered in
\cite{AbdaouiMabroukMakhlouf}
Generalized derivations and Rota-Baxter operators of $n$-ary Hom-Nambu superalgebras have been considered in \cite{MabroukNcibSilvestrov2020:GenDerRotaBaxterOpsnaryHomNambuSuperalgs}.
A construction of $3$-Hom-Lie algebras based on $\sigma$-derivation and involution has been studied in \cite{AbramovSilvestrov:3homLiealgsigmaderivINvol}.
Multiplicative $n$-Hom-Lie color algebras have been considered in
\cite{BakyokoSilvestrov:MultiplicnHomLiecoloralg}.

In \cite{almy:quantnambu}, Awata, Li, Minic and Yoneya introduced a construction of $(n+1)$-Lie algebras induced by $n$-Lie algebras using combination of bracket multiplication with a trace in their work on quantization of the Nambu brackets. Further properties of this construction, including solvability and nilpotency, were studied in \cite{BaiRBaiCWang:realizations3Liealg,akms:ternary,km:nary}.
In \cite{ams:ternary,ams:n}, this construction was generalized using the brackets of general Hom-Lie algebra or $n$-Hom-Lie and trace-like linear forms satisfying conditions depending on the twisting linear maps defining the Hom-Lie or $n$-Hom-Lie algebras.
In \cite{AbdelkhaderSamiOthmen}, a method was demonstrated of how to construct $n$-ary multiplications from the binary multiplication of a Hom-Lie algebra and a $(n-2)$-linear function satisfying certain compatibility conditions. Solvability and nilpotency for $n$-Hom-Lie algebras and $(n+1)$-Hom-Lie algebras induced by $n$-Hom-Lie algebras have been considered in \cite{kms:solvnilpnhomlie2020}.
In \cite{cln1dimnHomLieniltw}, properties and classification of $n$-Hom-Lie algebras in dimension $n+1$ were considered, and $4$-dimensional $3$-Hom-Lie algebras for various special cases of the twisting map have been computed in terms of structure constants as parameters and listed in classes in the way emphasizing the number of free parameters in each class.

The $n$-Hom-Lie algebras are fundamentally different from the $n$-Lie algebras especially when the twisting maps are not invertible or not diagonalizable. When the twisting maps are not invertible, the Hom-Nambu-Filippov identity becomes less restrictive since when elements of the kernel of the twisting maps are used, several terms or even the whole identity might vanish. Isomorphisms of Hom-algebras are also different from isomorphisms of algebras since they need to intertwine not only the multiplications but also the twisting maps. All of this make the classification problem different, interesting, rich and not simply following from the case of $n$-Lie algebras. In this work, we consider $n$-Hom-Lie algebras with a nilpotent twisting map $\alpha$, which means in particular that $\alpha$ is not invertible.

To our knowledge, the classification of $4$-dimensional $3$-Hom-Lie algebras up to Hom-algebras isomorphism has not been achieved previously in the literature. The aim of this work is to investigate the properties and classification of an interesting class of $4$-dimensional $3$-Hom-Lie algebras with a nilpotent twisting map $\alpha$ and eight structure constants as parameters, namely $4_{3,N(2),6}$ given in \cite{cln1dimnHomLieniltw}. All $3$-dimensional $3$-Hom-Lie algebras with diagonal twisting maps have been listed unclassified in \cite{AtMaSi:GenNambuAlg}.
The algebras considered in our article are $4$-dimensional, and the twisting maps are of a different type, namely nilpotent. Nilpotent linear maps are neither invertible nor diagonalizable, which makes the object of our study fundamentally different
from the case of $n$-Hom-Lie algebras with diagonal twisting maps in the sense that when the twisting maps are not invertible, the Hom-Nambu-Filippov identity becomes less restrictive since when elements of the kernel of the twisting maps are used in the identity, several terms or even the whole identity might vanish, and when the twisting maps are not diagonalizable, the change induced by introducing them in the identity is more significant.
In this work, we achieved a complete classification up to isomorphism of Hom-algebras of the considered class of $4$-dimensional $3$-Hom-Lie algebras with a nilpotent twisting map, computed derived series and central descending series for all of the $3$-Hom-Lie algebras of this class, studied solvability and nilpotency, characterized the multiplicative $3$-Hom-Lie algebras among them and studied the ideal properties of the terms of derived series and central descending series of some chosen examples of the Hom-algebras from the classification.  These results improve understanding of the rich structure of $n$-ary Hom-algebras and in particular the important class of $n$-Hom-Lie algebras.
It is also a step towards the complete classification of $4$-dimensional $3$-Hom-Lie algebras and in general $(n+1)$-dimensional $n$-Hom-Lie algebras.
Moreover, our results contribute to in-depth study of the structure and important properties and sub-classes of $n$-Hom-Lie algebras.

In Section \ref{sec:basicdefprelHomLie}, definitions and properties of $n$-Hom-Lie algebras that are used in the study are recalled, and new results characterizing nilpotency as well as necessary and sufficient conditions for multiplicativity of general $(n+1)$-dimensional $n$-Hom-Lie algebras and for algebras in the considered class are obtained in terms of the structure constants and the twisting map. In Section \ref{sec:Der_and_CD_series}, Derived series and central descending series are studied for all algebras in this class and are used to divide it into five non-isomorphic subclasses. The levels of solvability and nilpotency of the $3$-Hom-Lie algebras in these five classes are obtained. In Section \ref{sec:IsomorphismClasses}, building up on the previous sections, all algebras of this class are classified up to Hom-algebra isomorphism. In Section \ref{sec:examples}, for some algebras in this class, it has been determined whether the terms of the derived and central descending series are weak subalgebras, Hom-subalgebras, weak ideals or Hom-ideals.


\section{Definitions and properties of n-Hom-Lie algebras} \label{sec:basicdefprelHomLie}
In this section, we present the basic definitions and properties of $n$-Hom-Lie algebras needed for our study. Throughout this article, it is assumed that all linear spaces are over a field $\mathbb{K}$ of characteristic $0$, and for any subset $S$ of a linear space, $\langle S \rangle$ denotes the linear span of $S$.  The arity of all the considered algebras is assumed to be greater than or equal to $2$. Hom-Lie algebras are a generalization of Lie algebras introduced in \cite{HLS} while studying $\sigma$-derivations. The $n$-ary case was introduced in \cite{AtMaSi:GenNambuAlg}.

\begin{definition}[\cite{HLS,ms:homstructure}]
A Hom-Lie algebra $(A,[\cdot,\cdot],\alpha)$ is a linear space $A$ together with a bilinear map $[\cdot,\cdot] : A\times A \to A$ and a linear map $\alpha : A \to A$ satisfying, for all $x,y,z \in A$:
\begin{align*}
[x,y] &= - [y,x], && \hspace{-0,5cm} \text{Skew-symmetry} \\
\sum_{\circlearrowleft{(x,y,z)}} [\alpha(x),[y,z]] &= [\alpha(x),[y,z]] + [\alpha (y), [z,x]]+ [\alpha (z),[x,y]]= 0. &&  \begin{array}[c]{c} \text{Hom-Jacobi} \\ \text{identity } \\ \text{(cyclic form)}\end{array} \end{align*}
In Hom-Lie algebras, by skew-symmetry, the Hom-Jacobi identity is equivalent to
\begin{equation}  \label{HomLiejacobiHomderivform}
[\alpha(x),[y,z]] = [[x,y],\alpha (z)] + [\alpha (y), [x,z]] \quad   \begin{array}[t]{c} \text{Hom-Jacobi identity} \\ \text{(Hom-derivation form)}\end{array}
\end{equation}
\end{definition}

Hom-algebras satisfying just the Hom-algebra identity \eqref{HomLiejacobiHomderivform}, without requiring the skew-symmetry identity, are called Hom-Leibniz algebras
\cite{LS2,ms:homstructure}.
Thus, Hom-Lie algebras are skew-symmetric Hom-Leibniz algebras.
There are many Hom-Leibniz algebras which are not skew-symmetric and thus not Hom-Lie algebras.
When the twisting map is the identity map $\alpha=\Id_A$ on $A$, Hom-Leibniz algebras become (left) Leibniz algebras, and Hom-Lie algebras become Lie algebras.
A Hom-Leibniz algebra is also a Leibniz algebra,
or a Hom-Lie algebra is also a Lie algebra, if and only if
the map $\Id_A$ belongs to the set of all linear maps $\alpha$ for which the identity \eqref{HomLiejacobiHomderivform} holds.
Whether the map $\Id_A$ belongs to the set of all linear maps $\alpha$ for which the identity \eqref{HomLiejacobiHomderivform} holds or not depends on the underlying algebra. The Hom-algebra identity \eqref{HomLiejacobiHomderivform} is linear with respect to $\alpha$ in the linear space of all linear maps on the algebra, and hence, the set of all such $\alpha$, for which the identity \eqref{HomLiejacobiHomderivform} holds, is a linear subspace of the linear space of all linear maps on the algebra. There are many Hom-Leibniz algebra which are not Leibniz algebras, or Hom-Lie algebras which are not Lie algebras.

\begin{definition}[\cite{HLS,LarssonSilvJA2005:QuasiHomLieCentExt2cocyid}]
Hom-Lie algebra morphisms from Hom-Lie algebra $\mathcal{A}=(A,[\cdot,\cdot]_\mathcal{A},\alpha)$ to Hom-Lie algebra $\mathcal{B}=(B,[\cdot,\cdot]_\mathcal{B},\beta)$ are linear maps $f : A \to B$ satisfying, for all $x,y \in A$,
\begin{eqnarray}
f([x,y]_\mathcal{A}) &=& [f(x),f(y)]_\mathcal{B}, \label{homalgebramorphism} \\
f\circ \alpha &=& \beta \circ f. \label{homtwistsintertwining}
\end{eqnarray}
Linear maps $f : A \to B$ satisfying only condition \eqref{homalgebramorphism} are called weak morphisms of Hom-Lie algebras.
\end{definition}

\begin{definition}[\cite{BenMakh:Hombiliform,ms:homstructure}]
A Hom-Lie algebra $(A,[\cdot,\cdot],\alpha)$ is said to be multiplicative if $\alpha$ is an algebra morphism, and it is said to be regular if $\alpha$ is an isomorphism.
\end{definition}

\begin{definition}
An $n$-ary Hom-algebra $(A,[\cdot,\dots,\cdot],\{\alpha_i\}_{1 \leq i \leq n-1})$ is a linear space $A$ together with an $n$-ary operation, that is an $n$-linear map $[\cdot,\dots,\cdot]: A^n \to A$ and $(n-1)$ linear maps $\alpha_i : A \to A, 1 \leq i \leq n-1$.  An $n$-ary Hom-algebra is said to be skew-symmetric if its $n$-ary operation is skew-symmetric, that is satisfying for all $x_1,\dots,x_{n-1},y_1,\dots,y_n \in A$,
\begin{align}
& [x_{\sigma(1)},\dots,x_{\sigma(n)}] = sgn(\sigma) [x_1,\dots,x_n].
 & \text{Skew-symmetry}  \label{skewsym:nary}
\end{align}
\end{definition}

The $n$-Hom-Lie algebras are an $n$-ary generalization of Hom-Lie algebras to $n$-ary algebras satisfying a
generalisation of the Hom-algebra identity \eqref{HomLiejacobiHomderivform} involving $n$-ary product and $n-1$ linear maps.

\begin{definition}[\cite{AtMaSi:GenNambuAlg}]  \label{def:nhomLie}
An $n$-Hom-Lie algebra $(A,[\cdot,\dots,\cdot],\{\alpha_i\}_{1 \leq i \leq n-1})$ is a skew-symmetric $n$-ary Hom-algebra satisfying, for all $x_1,\dots,x_{n-1}, y_1,\dots,y_n \in A$
\begin{equation}
\begin{array}{l}
\text{Hom-Nambu-Filippov identity}\\[2mm]
\left[\alpha_1(x_1),\dots,\alpha_{n-1}(x_{n-1}),[y_1,\dots,y_n]\right] = \\*[0,1cm]
\quad \displaystyle{\sum_{i=1}^n} [\alpha_1(y_1),\dots,\alpha_{i-1}(y_{i-1}),[x_1,\dots,x_{n-1},y_i],
 \alpha_{i}(y_{i+1}),\dots,\alpha_{n-1}(y_{n})].
\end{array}
\label{Hom-Nambu-Filippov}
\end{equation}
\end{definition}

\begin{remark} If $\alpha_i=Id_A$ for all $1\leq i \leq n-1$, then one gets an $n$-Lie algebra (\cite{Filippov:nLie}). Therefore, the class of $n$-Lie algebras is included in the class of $n$-Hom-Lie algebras. For any linear space $A$, if  $\bracket{x_1,\dots,x_n}_0=0$ for all $x_1,\dots,x_n \in A$ and any linear maps $\alpha_1,\dots,\alpha_{n-1}$, then $(A, \bracket{\cdot,\dots,\cdot}_0,\alpha_1,\dots,\alpha_{n-1})$ is an $n$-Hom-Lie algebra.
\end{remark}

\begin{definition}[\cite{AtMaSi:GenNambuAlg,YauHomNambuLie}]
$n$-Hom-Lie algebra morphisms of $n$-Hom-Lie algebras $\mathcal{A}=(A,[\cdot,\dots,\cdot]_\mathcal{A},\{\alpha_i\}_{1 \leq i \leq n-1})$ and $\mathcal{B}=(B,[\cdot,\dots,\cdot]_\mathcal{B},\{\beta_i\}_{1 \leq i \leq n-1})$ are linear maps
$f : A \to B$ satisfying, for all $x_1,\dots,x_n \in A$,
\begin{eqnarray}
f([x_1,\dots,x_n]_\mathcal{A}) &=& [f(x_1),\dots,f(x_n)]_\mathcal{B},  \label{nhomalgebramorphism}\\
f\circ \alpha_i &=& \beta_i \circ f, \quad \text{for all} \quad  1 \leq i \leq n-1. \label{nhomtwistsintertwining}
\end{eqnarray}
Linear maps satisfying only condition \eqref{nhomalgebramorphism} are called weak morphisms of $n$-Hom-Lie algebras.
\end{definition}

The $n$-Hom-Lie algebras $(A,[\cdot,\dots,\cdot],\{\alpha_i\}_{1 \leq i \leq n-1})$ with  $\alpha_1=\dots=\alpha_{n-1}=\alpha$ will be denoted by $(A,\bracket{\cdot,\dots,\cdot},\alpha)$.
\begin{definition}[\cite{YauHomNambuLie}]
An $n$-Hom-Lie algebra $(A,\bracket{\cdot,\dots,\cdot},\alpha)$ is called multiplicative if $\alpha$ is an algebra morphism, and regular if $\alpha$ is an algebra isomorphism.
\end{definition}

The following proposition, providing a way to construct an $n$-Hom-Lie algebra from an $n$-Lie algebra and an algebra morphism, was first introduced in the case of Lie algebras and then generalized to the $n$-ary case in \cite{AtMaSi:GenNambuAlg}. A more general version of this theorem, given in \cite{YauHomNambuLie}, states that the category of $n$-Hom-Lie algebras is closed under twisting by weak morphisms.
\begin{proposition}[\cite{AtMaSi:GenNambuAlg, YauHomNambuLie}] \label{twist}
Let  $\beta : \mathcal{A} \to \mathcal{A}$ be a weak morphism of $n$-Hom-Lie algebra
$\mathcal{A}=\para{A,\bracket{\cdot,\dots,\cdot},\{\alpha_i\}_{1 \leq i \leq n-1}}$, and multiplication $\bracket{\cdot,\dots,\cdot}_\beta$ is defined by
$\bracket{x_1,\dots,x_n}_\beta = \beta\para{\bracket{x_1,\dots,x_n}}.$
Then, $\para{A,\bracket{\cdot,\dots,\cdot}_\beta,\{\beta\circ\alpha_i\}_{1 \leq i \leq n-1}}$ is an $n$-Hom-Lie algebra. Moreover, if $\para{A,\bracket{\cdot,\dots,\cdot},\alpha}$ is multiplicative and $\beta \circ \alpha = \alpha \circ \beta$, then $\para{A,\bracket{\cdot,\dots,\cdot}_\beta,\beta \circ \alpha}$ is multiplicative.
\end{proposition}

The following particular case of Proposition \ref{twist} is obtained if $\alpha=Id_A$.
\begin{corollary}
Let $\para{A,\bracket{\cdot,\dots,\cdot}}$ be an $n$-Lie algebra, $\beta : A \to A$ an algebra morphism, and $\bracket{\cdot,\dots,\cdot}_\beta$ is defined by
$ \bracket{x_1,\dots,x_n}_\beta = \beta\para{\bracket{x_1,\dots,x_n}}.$
Then, $\para{A,\bracket{\cdot,\dots,\cdot}_\beta,\beta}$ is a multiplicative $n$-Hom-Lie algebra.
\end{corollary}

The following definition is a specialization of the standard definition of a subalgebra in general algebraic structures to the case of $n$-Hom-Lie algebras and $n$-ary skew-symmetric Hom-algebras considered in this paper.
\begin{definition}
A Hom-subalgebra of an $n$-Hom-Lie algebra or more generally an $n$-ary skew-symmetric Hom-algebra $\mathcal{A}=(A,[\cdot,\dots,\cdot]_\mathcal{A},\alpha_1,\dots,\alpha_{n-1})$ is an $n$-ary Hom-algebra  $\mathcal{B}=(B,[\cdot,\dots,\cdot]_\mathcal{B},\beta_1,\dots,\beta_{n-1})$ consisting of
a subspace $B$ of $A$ satisfying, for all $x_1,\dots,x_n \in B$,
\begin{enumerate}[label=\textup{\arabic*)},ref=\textup{\arabic*}]
\item $\alpha_i(B) \subseteq B$ for all $1\leq i \leq n-1$,
\item $[x_1,\dots,x_n]_\mathcal{A}\in B$,
\end{enumerate}
with the restricted from $A$ multiplication $[\cdot,\dots,\cdot]_\mathcal{B}=[\cdot,\dots,\cdot]_\mathcal{A}$
and linear maps $\beta_i=\alpha_i, 1\leq i \leq n-1$ on $B$.
\end{definition}
The following definition is a direct extension of the corresponding definition in \cite{BenMakh:Hombiliform,ms:homstructure,YauHomNambuLie} to arbitrary $n$-ary skew-symmetric Hom-algebras.
\begin{definition}
An ideal of an $n$-Hom-Lie algebra or more generally of an $n$-ary skew-symmetric Hom-algebra $(A,[\cdot,\dots,\cdot],\alpha_1,\dots,\alpha_{n-1})$ is a subspace $I$ of $A$ satisfying, for all $x_1,\dots,x_{n-1} \in A$, $y \in I$:
\begin{enumerate}[label=\textup{\arabic*)},ref=\textup{\arabic*}]
\item $\alpha_i(I) \subseteq I$ for all $1\leq i \leq n-1$.
\item $[x_1,\dots,x_{n-1},y]\in I$ (or equivalently $[y,x_1,\dots,x_{n-1}]\in I$).
\end{enumerate}
\end{definition}
The following definitions are a direct extension of the corresponding definitions in \cite{kms:solvnilpnhomlie2020} to arbitrary $n$-ary skew-symmetric Hom-algebras.
\begin{definition}
Let $(A,\bracket{\cdot,\dots,\cdot},\alpha_1,\dots,\alpha_{n-1})$ be an $n$-Hom-Lie algebra or more generally an $n$-ary skew-symmetric Hom-algebra, and let $I$ be an ideal of $A$. For $2\leq k \leq n$ and $p\in \mathbb{N}$, we define the $k$-derived series of the ideal $I$ by
\[ D_k^0(I) = I \text{ and } D_k^{p+1}= \langle \big[\underset{k}{\underbrace{D_k^p(I),\dots,D_k^p(I)}},\underset{n-k}{\underbrace{A,\dots,A}}\big]\rangle. \]
We define the $k$-central descending series of $I$ by
\[C_k^0(I)=I \text{ and } C_k^{p+1}(I)=\langle \big[C_k^p(I),\underset{k-1}{\underbrace{I,\dots,I}},\underset{n-k}{\underbrace{A,\dots,A}}\big]\rangle.\]
\end{definition}

\begin{definition}
Let $(A,\bracket{\cdot,\dots,\cdot},\alpha_1,\dots,\alpha_{n-1})$ be an $n$-Hom-Lie algebra or more generally an $n$-ary skew-symmetric Hom-algebra, and let $I$ be an ideal of $A$. For $2\leq k \leq n$, the ideal $I$ is said to be $k$-solvable (resp. $k$-nilpotent) if there exists $r\in \mathbb{N}$ such that $D_k^r(I)=\{0\}$ (resp. $C_k^r(I)=\{0\}$),  and the smallest $r\in \mathbb{N}$ satisfying this condition is called  the class of $k$-solvability (resp. the class of nilpotency) of $I$.
\end{definition}

The following direct extension of the corresponding result in \cite{kms:solvnilpnhomlie2020} to arbitrary $n$-ary skew-symmetric Hom-algebras is proved in the same way as in \cite{kms:solvnilpnhomlie2020} since the proof does not involve the Hom-Nambu-Filippov identity.
\begin{lemma} \label{DCSurjMorphskwsym}
Let $\mathcal{A}=(A,\bracket{\cdot,\dots,\cdot}_A,(\alpha_i)_{1\leq i \leq n})$ and $\mathcal{B}=(B,\bracket{\cdot,\dots,\cdot}_B,(\beta_i)_{1\leq i \leq n})$ be two $n$-ary skew-symmetric Hom-algebras, $f : \mathcal{A} \to \mathcal{B}$ be a surjective $n$-Hom-Lie algebras morphism and $I$ an ideal of $\mathcal{A}$. Then for all $r\in \mathbb{N}$ and $2\leq k \leq n$:
\[ f\para{D_k^r(I)}=D_k^r\para{f(I)} \text{ and } f\para{C_k^r(I)}=C_k^r\para{f(I)}. \]
\end{lemma}

This lemma also implies that if two $n$-Hom-Lie algebras are isomorphic, they would also have isomorphic terms of the derived series and central descending series, which also means that if two algebras have a significant difference in the derived series or the central descending series, for example different dimensions of given corresponding terms, then these algebras cannot be isomorphic.

\begin{lemma}[\cite{cln1dimnHomLieniltw}]\label{HNFskew}
Let $A$ be a linear space, let $\bracket{\cdot,\dots,\cdot}$ be an $n$-linear skew-symmetric map ($n\geq 2$) and let $\alpha_1,\dots,\alpha_{n-1}$ be linear maps on $A$. If the $(n-1)$-linear map
\[(x_1,\dots,x_{n-1})\mapsto \bracket{\alpha_1(x_1),\dots,\alpha_{n-1}(x_{n-1}),d} \]
is skew-symmetric for all $d\in \bracket{A,\dots,A}$, then the $(2n-1)$-linear map $H$ defined by
\begin{align*}
 H&(x_1,\dots,x_{n-1},y_{1},\dots,y_n)=[\alpha_1(x_1),\dots,\alpha_{n-1}(x_{n-1}),[y_1,\dots,y_n]]\\
 &\quad - \sum_{k=1}^n [\alpha_1(y_1),\dots,\alpha_{k-1}(y_{k-1}),[x_1,\dots,x_{n-1},y_k], \alpha_{k}(y_{k+1}),\dots,\alpha_{n-1}(y_{n})],
\end{align*}
for all $x_1,\dots,x_{n-1},y_1,\dots,y_n\in A$, is skew-symmetric in its first $n-1$ arguments and in its last $n$ arguments.
\end{lemma}

\begin{proposition}[\cite{cln1dimnHomLieniltw}]
Let $A$ be an $n$-dimensional linear space ($n \geq 2$), and $(e_i)_{1\leq i \leq n}$ a basis of $A$. Any skew-symmetric $n$-linear map $\bracket{\cdot,\dots,\cdot}$ on $A$ is fully defined by giving $\bracket{e_1,\dots,e_n}=d\in A$. Let $\alpha_1,\dots,\alpha_{n-1}$ be linear maps on $A$. If the $(n-1)$-linear map
\[(x_1,\dots,x_{n-1})\mapsto \bracket{\alpha_1(x_1),\dots,\alpha_{n-1}(x_{n-1}),d} \]
is skew-symmetric, then $(A,\bracket{\cdot, \dots, \cdot},\alpha_1,\dots,\alpha_{n-1})$ is an $n$-Hom-Lie algebra.
\end{proposition}

\begin{corollary}[\cite{cln1dimnHomLieniltw}]\label{ndimensional}
Let $A$ be an $n$-dimensional linear space ($n \geq 2$), and $(e_i)_{1\leq i \leq n}$ a basis of $A$. Any skew-symmetric $n$-linear map $\bracket{\cdot,\dots,\cdot}$ on $A$ is fully defined by giving $\bracket{e_1,\dots,e_n}=d\in A$. For any linear map $\alpha$ on $A$, $(A,\bracket{\cdot,\dots,\cdot}, \alpha)$ is an $n$-Hom-Lie algebra.
\end{corollary}

Let $(A,\bracket{\cdot,\dots,\cdot},\alpha)$ be an $n$-ary skew-symmetric algebra of dimension $n+1$ with a linear map $\alpha$. Given a linear basis $(e_i)_{1\leq i \leq n+1}$ of $A$, the linear map $\alpha$ is fully determined by its matrix determined by action of $\alpha$ on the basis, and a skew-symmetric $n$-ary multi-linear bracket is fully determined by $\bracket{e_1,\dots,\widehat{e_i},\dots,e_{n+1}}$ for all $1\leq i \leq n+1$ represented by a matrix $B=(b(i,j))_{1\leq i,j \leq n+1}$ as follows:
\begin{eqnarray} \label{bracketmatrix}
\bracket{e_1,\dots,\widehat{e_i},\dots,e_{n+1}} &=& (-1)^{n+1+i}w_i,  \\
w_i = \sum_{p=1}^{n+1} b(p,i)e_p, \ (w_1,\dots,w_{n+1})&=& (e_1,\dots,e_{n+1})B.
\nonumber
\end{eqnarray}

\begin{proposition}[\cite{cln1dimnHomLieniltw}]
Let $\mathcal{A}_1=(A,\bracket{\cdot,\dots,\cdot}_1,\alpha_1)$ and $\mathcal{A}_2=(A,\bracket{\cdot,\dots,\cdot}_2,\alpha_2)$ be two $(n+1)$-dimensional $n$-ary skew-symmetric Hom-algebras represented by matrices $\bracket{\alpha_1}$, $B_1$ and $\bracket{\alpha_2}$, $B_2$ respectively. The Hom-algebras $\mathcal{A}_1$ and $\mathcal{A}_2$ are isomorphic if and only if there exists an invertible matrix $T$ satisfying the following conditions:
\begin{align*}
B_2 =\det(T)^{-1} T B_1 T^T,\quad \quad [\alpha_2] =T[\alpha_1]T^{-1}.
\end{align*}
\end{proposition}

\begin{proposition}[\cite{cln1dimnHomLieniltw}]\label{BasisPermutation}
Let $(e_i)_{1\leq i \leq n+1}$ be a basis of a linear space $A$, let $\sigma$ be a permutation of the set $\{1,\dots,n+1\}$ of $n+1$ elements, and let $B=(b_{i,j})_{1\leq i,j \leq n+1}$ be a matrix representing a skew-symmetric $n$-ary bracket in this basis, then the matrix representing the same bracket in the basis $(e_{\sigma(i)})_{1\leq i \leq n+1}$ is given by the matrix $sgn(\sigma)(b_{\sigma^{-1}(i),\sigma^{-1}(j)})_{1\leq i,j \leq n+1}$.
\end{proposition}

\begin{remark} {\rm (\cite{cln1dimnHomLieniltw}).}\label{DerivedSeriesIncompleteRank}
Let $(A,\bracket{\cdot,\dots,\cdot},\alpha)$ be an $(n+1)$-dimensional $n$-Hom-Lie algebra and let $B$ be the matrix representing its bracket. $D^1_n(A)=\bracket{A,\dots,A}$ is generated by $\{w_1,\dots, w_{n+1}\}$. Which means that $Rank(B)=\dim D^1_n(A)$.

If $Rank(B)\leq n$ or equivalently $\det(B)=0$ then $D^1_n(A)$ has dimension at most $n$, which means that $D^2_n(A)$ has dimension at most $1$ and then $D^3_n(A)=0$.
\end{remark}

\begin{remark} {\rm (\cite{cln1dimnHomLieniltw}).}
For the whole algebra $A$, all the $k$-central descending series,  for all $2\leq k \leq n$, are equal. Therefore all the notions of $k$-nilpotency, for all $2\leq k \leq n$, are equivalent, and we denote $C^p_k(A)$ for any $2\leq k \leq n$ by $C^p(A)$.
\end{remark}

\begin{definition}
Let $(A,\bracket{\cdot,\dots,\cdot},\alpha_1,\dots,\alpha_{n-1})$ be an $n$-Hom-Lie algebra or more generally an $n$-ary skew-symmetric Hom-algebra. Define $Z(A)$, the center of $A$, by
\begin{align*}
Z(A)=\{ z \in A : \bracket{x_1,\dots,x_{n-1},z}=0, \forall x_1,\dots,x_{n-1}\in A \}.
\end{align*}
\end{definition}

\begin{lemma}[\cite{cln1dimnHomLieniltw}]\label{nilcenter}
Let $(A,\bracket{\cdot,\dots,\cdot},\alpha)$ be an $n$-Hom-Lie algebra with $A\neq \{0\}$. If $A$ is $k$-nilpotent, for any $2\leq k \leq n$, then the center $Z(A)$ of $A$  is not trivial ($Z(A)\neq \{ 0 \}$).
\end{lemma}

\begin{lemma}  \label{lem:nildimcenter}
Let $\mathcal{A}=(A,\bracket{\cdot,\dots,\cdot},(\alpha_i)_{1\leq  i \leq n-1})$ be an $n$-ary skew-symmetric Hom-algebra with $A\neq \{0\}$.
\begin{enumerate}[label=\textup{(\roman*)},ref=\textup{(\roman*)}]
\item \label{itemi:lem:nildimcenter} If $\mathcal{A}$, is nilpotent then $Z(\mathcal{A})$ is not trivial ($Z(A)\neq \{ 0 \}$).
\item \label{itemii:lem:nildimcenter} If $\dim A = n+1$, then $\dim Z(\mathcal{A})=0$ or $\dim Z(\mathcal{A})=1$ or $Z(\mathcal{A})=A$.
\end{enumerate}
\end{lemma}

\begin{proof}
\noindent \ref{itemi:lem:nildimcenter}  The first statement is a generalization of Lemma \ref{nilcenter} to the case of $n$-ary skew-symmetric Hom-algebras, and is proved in the same way, since the original proof does not use the Hom-Nambu-Filippov identity.  \\
\noindent \ref{itemii:lem:nildimcenter} Suppose that $\dim A = n+1$ and that $\dim Z(\mathcal{A}) >1$. Let $(e_i)$ be a basis of $A$ such that $e_1,e_2 \in Z(\mathcal{A}$, then $\bracket{e_1,\dots,\widehat{e_i},\dots,e_{n+1}  }=0$ for all $1\leq i \leq n+1$, which means that $\bracket{x_1,\dots,x_n}=0$ for all $x_1,\dots,x_n \in A$.\qed
\end{proof}

The following direct extension of the corresponding result in \cite{cln1dimnHomLieniltw} to arbitrary $n$-ary skew-symmetric Hom-algebras is proved in the same way as in \cite{cln1dimnHomLieniltw} since the proof does not involve the Hom-Nambu-Filippov identity.
\begin{proposition}\label{centernil_n+1}
Let $\mathcal{A}=(A,[\cdot,\dots,\cdot],\{\alpha_i\}_{1 \leq i \leq n-1})$ be an $(n+1)$-dimensional $n$-ary skew-symmetric algebra. The algebra $\mathcal{A}$ is nilpotent and non abelian if and only if $\dim Z(\mathcal{A}) = 1$ and $\bracket{A,\dots,A} = Z(\mathcal{A})$.
\end{proposition}

\begin{proposition}
Let $\mathcal{A}=(A,[\cdot,\dots,\cdot],\{\alpha_i\}_{1 \leq i \leq n-1})$ be an $n$-Hom-Lie algebra or more generally an $n$-ary skew-symmetric Hom-algebra with $A\neq \{0\}$. $\mathcal{A}$ is nilpotent of class $p$ if and only if $\{0\} \subsetneq C^{p-1}(A)\subseteq Z(A)$.
\end{proposition}

\begin{proof} The statement holds,
since $\mathcal{A}$ is nilpotent of class $p$ if and only if $C^p(A)=\{0\}$ and $C^{p-1}(A)\neq \{0\}$, and
\begin{align*}
C^p(A)=\{0\} & \iff \bracket{C^{p-1}(A),A,\dots,A}=\{0\}\\
& \iff \forall\    c \in C^{p-1}(A), \forall\    x_1,\dots,x_{n-1}\in A, \bracket{c,x_1,\dots,x_{n-1}}=0\\
& \iff \forall\    c\in C^{p-1}(A), c\in Z(\mathcal{A})\\
& \iff C^{p-1}(A) \subseteq Z(\mathcal{A}).
\tag*{\qed} \end{align*}
\end{proof}

\begin{proposition}
Let $\mathcal{A}=(A,\bracket{\cdot,\dots,\cdot}_\mathcal{A},\alpha)$ and $\mathcal{B}=(B,\bracket{\cdot,\dots,\cdot}_\mathcal{B},\beta)$ be $n$-ary Hom-algebras. Let $f:\mathcal{A}\to \mathcal{B}$ be an $n$-ary Hom-algebras homomorphism, then if $\mathcal{A}$ is multiplicative then $\mathcal{B}$ is multiplicative. Moreover, if $f$ is an isomorphism, then $\mathcal{A}$ is multiplicative if and only if $\mathcal{B}$ is multiplicative.
\end{proposition}

\begin{proof}
Let $f:\mathcal{A}\to \mathcal{B}$ be a surjective homomorphism, then for all $y_1,\dots,y_n\in B$ there exists $x_1,\dots,x_n\in A$ such that $f(x_i)=y_i$ for $1\leq i \leq n$, and $\beta \circ f = f\circ \alpha$. Suppose that $\mathcal{A}$ is multiplicative, then we have,
\begin{align*}
\beta\left(\bracket{y_1,\dots,y_n}_\mathcal{B} \right)&=\beta\left(\bracket{f(x_1),\dots,f(x_n)}_\mathcal{B} \right)\\
&=\beta\circ f\left(\bracket{x_1,\dots,x_n}_\mathcal{A} \right)=f\circ \alpha \left(\bracket{x_1,\dots,x_n}_\mathcal{A} \right)\\
&=\bracket{f\circ \alpha (x_1),\dots,f\circ \alpha(x_n)}_\mathcal{B} =\bracket{\beta \circ f(x_1),\dots,\beta \circ f(x_n)}_\mathcal{B} \\
&=\bracket{\beta(y_1),\dots,\beta(y_n)}_\mathcal{B}.
\end{align*}
If $f$ is an isomorphism, then the converse can be proved by applying the same argument using $f^{-1}$ instead of $f$.
\end{proof}

\begin{proposition}[\cite{cln1dimnHomLieniltw}]
Let $(A,\bracket{\cdot,\dots,\cdot},\alpha)$ be an $n$-ary Hom-algebra with $\dim A = n+1$, $\bracket{\cdot,\dots,\cdot}$ skew-symmetric, $\alpha$ nilpotent, $\dim \ker \alpha = 2$ and the bracket is represented by the matrix $B=(b_{i,j})$ as in \eqref{bracketmatrix}, in a basis where $\alpha$ is in Jordan normal form. The bracket $\bracket{\cdot,\dots,\cdot}$ satisfies the Hom-Nambu-Filippov identity if and only if
\[b_{i_0-1,j} b_{p,n+1} - b_{n+1,j} b_{p,i_0-1}=0,\forall\    1\leq j,p\leq n+1, j\neq 1, j\neq i_0, \]
where $i_0$ is such that $\ker \alpha = \langle e_1, e_{i_0} \rangle$.
\end{proposition}

\begin{remark}
Let us compare the polynomial equations obtained from the Nambu-Filippov identity and the Hom-Nambu-Filippov identity in dimension $n+1$ with various types of twisting maps: \\
Diagonalizable and invertible with eigenvalues $\lambda_i, 1\leq i \leq n+1$:
\begin{align} \label{diag-inv}
 (\lambda_i b_{j,i} - \lambda_j b_{i,j}) b_{p,k}  +(\lambda_k b_{i,k} - \lambda_i b_{k,i}) b_{p,j}   + (\lambda_j b_{k,j} - \lambda_k b_{j,k} ) b_{p,i}=0,\\
\forall\   1\leq i,j,k,p \leq n+1; i<j<k; \nonumber
 \end{align}
Diagonalizable with $\dim \ker \alpha=1$ with eigenvalues $\lambda_i, 1\leq i \leq n+1$:
\begin{equation}\label{diag-ker1}
\lambda_k b_{1,k} w_j - \lambda_k b_{j,k} w_1 -\lambda_j b_{1,j}w_k  + \lambda_j b_{k,j}w_1=0, \quad \forall\   1<j<k \leq n+1;
\end{equation}
Diagonalizable with $\dim \ker \alpha=2$ with eigenvalues $\lambda_i, 1\leq i \leq n+1$:
\begin{equation}\label{diag-ker2}
 b_{1,k}  w_2   -  b_{2,k}  w_1=0, \quad \forall\   3\leq k \leq n+1;
 \end{equation}
Nilpotent with $\dim \ker \alpha=1$:
\begin{align}\label{nil-ker1}
  (b_{k-1,i}-b_{i-1,k}) b_{p,n+1} - b_{n+1,i}b_{p,k-1} + b_{n+1,k}b_{p,i-1}=0, \\
  \forall\   1\leq i,k,p\leq n+1, i<k;   \nonumber
\end{align}
Nilpotent with $\dim \ker \alpha=2$:
\begin{equation}\label{nil-ker2}
b_{i_0-1,j} b_{p,n+1} - b_{n+1,j} b_{p,i_0-1}=0, \quad \forall\   1\leq j,p\leq n+1, j\neq 1, j\neq i_0.
\end{equation}
These different cases are separate from each other, and the case of $n$-Lie algebras is the special case of \eqref{diag-inv} where all the $\lambda_i$ are equal. Notice that the higher the dimension of $\ker \alpha$ the less equation we have and the less terms we have in each equation, that is, in these cases, the Hom-Nambu-Filippov identity is considerably less restrictive. Another difference from the case of $n$-Lie algebras is that the isomorphisms in Hom-algebras intertwine the multiplications and the twisting maps, which leads to different, more restrictive isomorphism conditions and, in general, more isomorphism classes.
\end{remark}

\begin{lemma}\label{multiplicativity_kernel}
Let $(A,\bracket{\cdot,\dots,\cdot},\alpha)$ be an $n$-ary Hom-algebra with $\dim A = n+1$, $\bracket{\cdot,\dots,\cdot}$ skew-symmetric and $\alpha$ nilpotent. Let $(e_i)_{1\leq i \leq n+1}$ be a basis of $A$ where $\alpha$ is in its Jordan form, and consider $\bracket{\cdot,\dots,\cdot}$ to be defined as in \eqref{bracketmatrix}.

If $\dim \ker \alpha \geq 2$, then $(A,\bracket{\cdot,\dots,\cdot},\alpha)$ is multiplicative if and only if $\bracket{A,\dots,A}\subseteq \ker \alpha$.

If $\dim \ker \alpha=1$, then $(A,\bracket{\cdot,\dots,\cdot},\alpha)$ is multiplicative if and only if $\alpha(w_1)=(-1)^{n} w_{n+1}$ and $w_i \in \ker \alpha, \forall\   2\leq i \leq n+1$, where $(w_i)$ are defined in \eqref{bracketmatrix}.
\end{lemma}

\begin{proof}
Suppose that $\dim \ker \alpha \geq 2$, then for all $1\leq i \leq n+1$,
\begin{align*}
\alpha(w_i)&=(-1)^{n+1+i}\alpha\para{\bracket{e_1,\dots,\widehat{e_i},\dots,e_{n+1} }} \\
&=(-1)^{n+1+i}\bracket{\alpha\para{e_1},\dots,\widehat{\alpha\para{e_i} },\dots,\alpha\para{e_{n+1} }  }=0,
\end{align*}
since $e_i\in \ker \alpha$ for at least two different indices $i$, that is at least one of the $\alpha\para{e_1},\dots,\widehat{\alpha\para{e_i} },\dots,\alpha\para{e_{n+1} }$ is zero. Thus, $\bracket{A,\dots,A}=\langle w_1,\dots,w_{n+1}\rangle\subseteq \ker \alpha$.

Suppose now that $\dim \ker \alpha =1$, then we have $\alpha(e_1)=0$ and $\alpha(e_{i})=e_{i-1}$ for $2\leq i \leq n+1$. We get
\begin{align*}
\alpha(w_1)&=(-1)^{n+1+1}\alpha\para{\bracket{e_2,\dots,e_{n+1}} } = (-1)^{n}\bracket{\alpha(e_2),\dots,\alpha(e_{n+1}) }\\
&=(-1)^{n}\bracket{e_1,\dots,e_n}=(-1)^{n}(-1)^{n+1+n+1} w_{n+1}=(-1)^{n}w_{n+1}.
\end{align*}
For $i\neq 1$ we have,
\begin{align*}
\alpha(w_i) &= (-1)^{n+1+i}\alpha\para{\bracket{e_1,\dots,\widehat{e_i},\dots,e_{n+1} }} \\
&=(-1)^{n+1+i}\bracket{\alpha\para{e_1},\dots,\widehat{\alpha\para{e_i} },\dots,\alpha\para{e_{n+1} }  }\\
&=(-1)^{n+1+i}\bracket{0,e_1,\dots,\widehat{e_{i-1}},\dots,e_n }=0,
\end{align*}
that is $\alpha(w_{i})=0$ for $i\neq 1$.
\end{proof}

\begin{proposition}
Let $\mathcal{A}=(A,\bracket{\cdot,\dots,\cdot},\alpha)$ be an $(n+1)$-dimensional $n$-Hom-Lie algebra. If $\dim \ker \alpha \geq 2$ then $\mathcal{A}$ is multiplicative if and only if $[\alpha]B=0$, where $[\alpha]$ and $B$ are the matrices representing the twisting map $\alpha$ and the bracket in any given basis.
\end{proposition}

\begin{proof}
Let $(e_i)_{1\leq i \leq n+1}$ be a basis of $A$ containing a basis of $\ker \alpha$. Then $\mathcal{A}$ is multiplicative if and only if $$\alpha\left( \bracket{e_1,\dots,\widehat{e_i},\dots,e_{n+1}} \right) = \bracket{\alpha(e_1)
,\dots,\widehat{\alpha(e_i)},\dots,\alpha(e_{n+1})}
\ \mbox{for all} \ 1\leq i \leq n+1.$$ On the other hand,  $\bracket{\alpha(e_1),\dots,\widehat{\alpha(e_i)},\dots,\alpha(e_{n+1})}=0$ since  at least one of the elements $e_1,\dots,e_{i-1},e_{i+1},\dots,e_{n+1}$ is in $\ker \alpha$. Moreover $[\alpha]B$ is the matrix whose columns are the coordinates of $(-1)^{n+i+1}\alpha\left( \bracket{e_1,\dots,\widehat{e_i},\dots,e_{n+1}} \right)$. Thus $\alpha$ is an algebra morphism if and only if $[\alpha]B=0$.

Let now $[\alpha]_2$ and $B_2$ be the matrices representing $\alpha$ and $\bracket{\cdot,\dots,\cdot}$ in another basis $(e'_i)$, then there exists an invertible matrix $P$ such that $[\alpha]_2=P [\alpha] P^{-1}$ and $B_2=(\det P)^{-1} P B P^T$, and we get
\begin{align*}
[\alpha]_2 B_2 &=  (P [\alpha] P^{-1})((\det P)^{-1} P B P^T)\\
&=(\det P)^{-1} (P [\alpha] P^{-1} P B P^T) = (\det P)^{-1} (P [\alpha] B P^T).
\end{align*}
Therefore $[\alpha]_2 B_2=0$ if and only if $[\alpha] B=0$, since $P$ is invertible.
\end{proof}

\begin{corollary}
Let $(A,\bracket{\cdot,\dots,\cdot},\alpha)$ be an $n$-ary Hom-algebra with $\dim A = n+1$, $\bracket{\cdot,\dots,\cdot}$ skew-symmetric and $\alpha$ nilpotent. Let $(e_i)_{1\leq i \leq n+1}$ be a basis of $A$ where $\alpha$ is in its Jordan form, and consider $\bracket{\cdot,\dots,\cdot}$ to be defined by its structure constants in this basis, that is,
$ \bracket{e_{i_1},\dots,e_{i_n}}=\sum\limits_{k=1}^{\dim A} c_{i_1,\dots,i_n}^k e_k$.

If $\dim \ker \alpha \geq 2$, then $(A,\bracket{\cdot,\dots,\cdot},\alpha)$ is multiplicative if and only if $c_{i_1,\dots,i_n}^k=0$, for all $1\leq i_1,\dots,i_n\leq \dim A$ and $k$ such that $e_k \notin \ker \alpha$.

\begin{remark}
Note that when $\dim A = n+1$, it is sufficient to define the bracket by its structure constants as
$ \bracket{e_{1},\dots,\widehat{e_i},\dots,e_{n+1}}=\sum\limits_{k=1}^{\dim A} c_{1,\dots,i-1,i+1,\dots,n+1}^k e_k$.
The parameters $b(p,i)$ in \eqref{bracketmatrix} are
$b(p,i)=(-1)^{n+1+i} c_{1,...,i-1,i+1,...,n+1}^p$.
\end{remark}

\end{corollary}

\section{Class \texorpdfstring{$4_{3,N(2),6}$}{} of $4$-dimensional $3$-Hom-Lie algebras}
An interesting class of $4$-dimensional $3$-Hom-Lie algebras $4_{3,N(2),6}=(A,\bracket{\cdot,\dots,\cdot},\alpha)$
is defined according to \eqref{bracketmatrix}
on the basis $(e_i)_{1\leq i \leq 4}$ by \\
\begin{center}
$\bracket{\alpha}=
\begin{pmatrix}
0 & 0 & 0 & 0 \\
0 & 0 & 1 & 0 \\
0 & 0 & 0 & 1 \\
0 & 0 & 0 & 0 \\
\end{pmatrix},\quad B=\left(
\begin{array}{cccc}
 0 & c(1,3,4,1) & -c(1,2,4,1) & 0 \\
 0 & c(1,3,4,2) & -c(1,2,4,2) & 0 \\
 0 & c(1,3,4,3) & -c(1,2,4,3) & 0 \\
 0 & c(1,3,4,4) & -c(1,2,4,4) & 0 \\
\end{array}
\right),$
\begin{align*}
\bracket{e_1,e_2,e_3}&=0\\
\bracket{e_1,e_2,e_4}&=c(1,2,4,1) e_1 + c(1,2,4,2) e_2 + c(1,2,4,3) e_3 + c(1,2,4,4) e_4\\
\bracket{e_1,e_3,e_4}&=c(1,3,4,1) e_1 + c(1,3,4,2) e_2 + c(1,3,4,3) e_3 + c(1,3,4,4) e_4\\
\bracket{e_2,e_3,e_4}&=0,
\end{align*}
\end{center}
where $c(i_1,\dots,i_n,k)=c_{i_1,\dots,i_n}^k$ are the structure constants according to
$$\bracket{e_{i_1},\dots,e_{i_n}}=\sum\limits_{k=1}^{\dim A} c_{i_1,\dots,i_n}^k e_k=\sum\limits_{k=1}^{\dim A} c(i_1,\dots,i_n,k) e_k.$$

Applying Lemma \ref{multiplicativity_kernel} to the class of $3$-Hom-Lie algebras $4_{3,N(2),6}$, we get the following result describing all multiplicative $3$-Hom-Lie algebras in the class $4_{3,N(2),6}$.
\begin{corollary}\label{mul_sp}
The $3$-Hom-Lie algebra from $4_{3,N(2),6}$ is multiplicative if and only if \[c(1,2,4,3)=0,\ c(1,2,4,4) =0,\ c(1,3,4,3)=0,\ c(1,3,4,4) =0.\]
\end{corollary}
\begin{proof}
By Lemma \ref{multiplicativity_kernel}, the $3$-Hom-Lie algebra $4_{3,N(2),6}$ is multiplicative if and only if $\bracket{e_1,e_2,e_4}, \bracket{e_1,e_3,e_4} \in \ker \alpha$ which is $\langle \{e_1,e_2\}\rangle$, and this is the case if and only if $c(1,2,4,3)=0$, $c(1,2,4,4)=0$, $c(1,3,4,3)=0$, $c(1,3,4,4) =0.$
\qed
\end{proof}
So, the $3$-Hom-Lie algebra from $4_{3,N(2),6}$ is in the subclass $4_{3,N(2),6,M}$ of multiplicative $3$-Hom-Lie algebras, if and only if the multiplication (bracket) is defined by
\begin{align*}
\bracket{e_1,e_2,e_3}&=0,\\
\bracket{e_1,e_2,e_4}&=c(1,2,4,1) e_1 + c(1,2,4,2) e_2, \\
\bracket{e_1,e_3,e_4}&=c(1,3,4,1) e_1 + c(1,3,4,2) e_2, \\
\bracket{e_2,e_3,e_4}&=0.
\end{align*}


\section{Derived series and central descending series for \texorpdfstring{$4_{3,N(2),6}$}{}}\label{sec:Der_and_CD_series}

A consequence of Lemma \ref{DCSurjMorphskwsym} is that the derived series and the central descending series of an $n$-Hom-Lie algebra are algebraic invariants. Here, we divide the considered class of $3$-Hom-Lie algebras into five subclasses following their derived series and central descending series. Two $3$-Hom-Lie algebras in two different subclasses will necessarily be non-isomorphic, and we use this as an intermediate step towards the full classification up to isomorphism of the algebras in this class.

In the case of $n$-Hom-Lie algebras, the terms of the derived series and the central descending series are in general not ideals as in the case of $n$-Lie algebras. In the most general case, they are weak subalgebras, and they can be subalgebras or ideals if the twisting maps are algebra morphisms or surjective algebra morphisms respectively, as it has been shown in \cite{kms:solvnilpnhomlie2020}. For the case of $4_{3,N(2),6,M}$, we have the following result.

\begin{theorem}\label{5cases}
Consider $\mathcal{A}=(A,\bracket{\cdot,\cdot,\cdot},\alpha)=4_{3,N(2),6}$. Suppose that $B\neq 0$ and define $d(p,q)=c(1,2,4,p)c(1,3,4,q)-c(1,2,4,q)c(1,3,4,p)$ with $1\leq p,q \leq 4$, that is, $d(p,q)$ are all the potentially non-zero $2\times 2$ subdeterminants of the matrix $B$ defining the bracket of $\mathcal{A}$. Then $\mathcal{A}$ is $3$-solvable of class $2$.

$\mathcal{A}$ is $2$-solvable if and only if $d(1,4)=0$, this implies moreover that there exists $(\lambda,\lambda') \in \mathbb{K}^2\setminus \{(0,0)\}$ such that $\lambda d(2,4) + \lambda'd(1,2)=0$ and $\lambda d(3,4)+\lambda' d(1,3)=0$, or equivalently that $Rank \begin{pmatrix}
d(2,4) & d(3,4)\\
d(1,2) & d(1,3)
\end{pmatrix} <2$
 which is equivalent to $\begin{vmatrix}
d(2,4) & d(3,4)\\
d(1,2) & d(1,3)
\end{vmatrix}=0$.

If $Rank B = 2$ or equivalently, there exists $1\leq p < q \leq 4$ such that $d(p,q)\neq 0$, then
\begin{enumerate}[label=\textup{\arabic*)},ref=\textup{\arabic*}]
\item $Z(\mathcal{A})=\{0\}$. This also means that $4_{3,N(2),6}$ is not nilpotent.
\item If $\mathcal{A}$ is $2$-solvable, then
\begin{enumerate}
\item If  $\begin{pmatrix}
d(2,4) & d(3,4)\\
d(1,2) & d(1,3)
\end{pmatrix}\neq 0$, then $\mathcal{A}$ is $2$-solvable of class $3$.
\item
If $\begin{pmatrix}
d(2,4) & d(3,4)\\
d(1,2) & d(1,3)
\end{pmatrix}= 0$ then, $\mathcal{A}$ is $2$-solvable of class $2$.
\end{enumerate}
\end{enumerate}
If $Rank B =1$ or equivalently $d(p,q)=0$, for all $1\leq p < q \leq 4$, then $4_{3,N(2),6}$ is $2$-solvable of class $2$, and also $\dim Z(\mathcal{A}) =1$, and
$$
Z(\mathcal{A})=\langle \{ c(1,3,4,p) e_2 - c(1,2,4,p)e_3\} \rangle,
$$
where  $c(1,2,4,p)\neq 0$ or $c(1,3,4,p)\neq 0$. Moreover, the algebra is nilpotent if and only if $Z(\mathcal{A})=\bracket{A,A,A}$, or equivalently if and only if $c(1,2,4,1)=c(1,2,4,4)=c(1,3,4,1)=c(1,3,4,4)=0$ and $c(1,3,4,p)c(1,2,4,3)+c(1,2,4,p)c(1,2,4,2)=0$ and $c(1,3,4,p)c(1,3,4,3)+c(1,2,4,p)c(1,3,4,2)=0$.

\end{theorem}
\begin{proof}
By Remark \ref{DerivedSeriesIncompleteRank}, we know that $4_{3,N(2),6}$ is $3$-solvable. The derived series of $\mathcal{A}$ are given by
\begin{align*}
D^1_3(\mathcal{A})&=\langle \{ c(1,2,4,1) e_1 + c(1,2,4,2) e_2 + c(1,2,4,3) e_3 + c(1,2,4,4) e_4,  \\
&\quad c(1,2,4,1) e_1+ c(1,2,4,2) e_2+ c(1,2,4,3) e_3 + c(1,2,4,4) e_4 \} \rangle,
\end{align*}
 and
$D^2_3(\mathcal{A})=\bracket{D^1_3(\mathcal{A}),D^1_3(\mathcal{A}),D^1_3(\mathcal{A})}=\{0\}$ by skew-symmetry, since $\dim D^1_3(A)$ is less than $3$ (the arity).
 We compute now the $2$-derived series,
\begin{align*}
D^1_2(\mathcal{A})&=\langle \{ c(1,2,4,1) e_1 + c(1,2,4,2) e_2 + c(1,2,4,3) e_3 + c(1,2,4,4) e_4,\\
&\quad  c(1,3,4,1) e_1 + c(1,3,4,2) e_2 + c(1,3,4,3) e_3 + c(1,3,4,4) e_4 \} \rangle
\end{align*}
We have $0\leq \dim D^1_2(A)\leq 2$. If $\dim D^1_2(A)=2$, then
\begin{align}\label{D22}
D^2_2(\mathcal{A})&=\langle \{ \bracket{e_1,w_2,w_3}, \bracket{e_2,w_2,w_3}, \bracket{e_3,w_2,w_3}, \bracket{e_4,w_2,w_3} \} \rangle\\
&=\langle \{
(c(1,3,4,2)c(1,2,4,4)-c(1,3,4,4)c(1,2,4,2))w_3 \nonumber \\
&\hspace{3cm}-(c(1,3,4,3)c(1,2,4,4)-c(1,3,4,4)c(1,2,4,3))w_2,\nonumber\\
& \quad-(c(1,3,4,1)c(1,2,4,4)-c(1,3,4,4)c(1,2,4,1))w_3,\nonumber\\
& \quad-(c(1,3,4,1)c(1,2,4,4)-c(1,3,4,4)c(1,2,4,1))w_2,\nonumber\\
& \quad(c(1,3,4,1)c(1,2,4,2)-c(1,3,4,2)c(1,2,4,1))w_3 \nonumber\\
&\hspace{2.7cm} -(c(1,3,4,1)c(1,2,4,3)-c(1,3,4,3)c(1,2,4,1))w_2\} \rangle.\nonumber
\end{align}
If $\dim D^2_2(\mathcal{A})=2$, then $D^2_2(\mathcal{A})=D^1_2(\mathcal{A})$ since $D^2_2(\mathcal{A})\subseteq D^1_2(\mathcal{A})$ and has the same dimension. We conclude in this case that $A$ is not $2$-solvable.

If $\dim D^2_2(\mathcal{A})=1$, then $D^2_2(\mathcal{A})=\langle \{v\} \rangle$ with $v\in A, v\neq 0$. In this case, $D^3_2(\mathcal{A})=\langle \{ [e_i,v,v], 1\leq i\leq 4 \} \rangle$, that is $D^3_2(\mathcal{A})=\{0\}$ and $A$ is $2$-solvable of class $3$.
This occurs if and only if the rank of the family of generators of $D_2^2(\mathcal{A})$ listed in \eqref{D22} is $1$, that is if and only if,
for some $\lambda,\lambda' \in \mathbb{K}$,
\begin{align*}
&(c(1,3,4,1)c(1,2,4,4)-c(1,3,4,4)c(1,2,4,1))=0, \\
\lambda &(c(1,3,4,2)c(1,2,4,4)-c(1,3,4,4)c(1,2,4,2))\\
&+  \lambda' (c(1,3,4,1)c(1,2,4,2)-c(1,3,4,2)c(1,2,4,1))=0,\\
\lambda & (c(1,3,4,3)c(1,2,4,4)-c(1,3,4,4)c(1,2,4,3)) \\
&+\lambda' (c(1,3,4,1)c(1,2,4,3)-c(1,3,4,3)c(1,2,4,1))=0.\end{align*}

On the other hand, we have that
\begin{align*}
\det \begin{pmatrix}
d(2,4) & d(3,4)\\
d(1,2) & d(1,3)
\end{pmatrix}&=(c(1,2,4,3) c(1,3,4,2)-c(1,2,4,2) c(1,3,4,3))\times\\
&\qquad \times (c(1,2,4,4) c(1,3,4,1)-c(1,2,4,1) c(1,3,4,4)\\
&=d(2,3)d(1,4),
\end{align*}
which means that $\det \begin{pmatrix}
d(2,4) & d(3,4)\\
d(1,2) & d(1,3)
\end{pmatrix}=0$ if and only if $d(2,3)=0$ or $d(1,4)=0$. This means also that the condition $\det \begin{pmatrix}
d(2,4) & d(3,4)\\
d(1,2) & d(1,3)
\end{pmatrix}=0$ and $d(1,4)=0$ is equivalent to only saying that $d(1,4)=0$.

The coefficients appearing in the generators of $D^2_2(A)$ in \eqref{D22} are the entries of the matrix $\begin{pmatrix}
d(2,4) & d(3,4)\\
d(1,2) & d(1,3)
\end{pmatrix}$, that is
$D^2_2(A)=\{0\}\ \text{if and only if}\
\begin{pmatrix}
d(2,4) & d(3,4)\\
d(1,2) & d(1,3)
\end{pmatrix}=0.$

If $\dim D^1_2(A) = 1$, then all the coefficients appearing in the generators of $D^2_2(A)$ are zero, since they are $2\times 2$ subdeterminants of the matrix $B$ which is of rank $1$. This means that $D^2_2(A)=\{0\}$ and $A$ is $2$-solvable of class $2$.

We know that an $(n+1)$-dimensional $n$-Hom-Lie algebra is nilpotent and non-abelian, if and only if $[A,\dots,A]=Z(A)$ and $\dim Z(A)=1$ (See \cite{cln1dimnHomLieniltw} Proposition 9). Therefore, if $\dim [A,\dots,A]=2$, $A$ cannot be nilpotent. In this case $C_k^r(A)=\langle \{w_2,w_3\} \rangle$,  for all $r\geq 1$. Consider now the center of $\mathcal{A}$,
\begin{align*}
Z(A)&=\{ z=\sum_{k=1}^4 z_k e_k \mid \forall\   x,y \in A, \bracket{x,y,z}=0 \}\\
&= \{ z=\sum_{k=1}^4 z_k e_k \mid \forall\   1\leq i < j \leq 4, \bracket{e_i,e_j,z}=0 \}
\end{align*}
and we get the following system of equations
\begin{align*}
c(1,2,4,1)z_1=0 ;\  c(1,2,4,2)z_1=0 ;\  c(1,2,4,3)z_1=0 ;\  c(1,2,4,4)z_1=0;\ \\
c(1,3,4,1)z_1=0 ;\  c(1,3,4,2)z_1=0 ;\  c(1,3,4,3)z_1=0 ;\  c(1,3,4,4)z_1=0;\ \\
c(1,2,4,1)z_2+c(1,3,4,1)z_3=0;\  c(1,2,4,2)z_2+c(1,3,4,2)z_3=0;\  \\
c(1,2,4,3)z_2+c(1,3,4,3)z_3=0;\  c(1,2,4,4)z_2+c(1,3,4,4)z_3=0;\ \\
c(1,2,4,1)z_4=0 ;\  c(1,2,4,2)z_4=0 ;\  c(1,2,4,3)z_4=0 ;\  c(1,2,4,4)z_4=0;\ \\
c(1,3,4,1)z_4=0 ;\  c(1,3,4,2)z_4=0 ;\  c(1,3,4,3)z_4=0 ;\  c(1,3,4,4)z_4=0.
\end{align*}
Then we get, $z_1\neq 0$ or $z_4\neq 0$ if and only if the algebra is abelian, that is $c(1,2,4,i)=c(1,3,4,i)=0$, for all $1\leq i \leq 4$. Excluding this case, we get the following system
\begin{align*}
c(1,2,4,1)z_2+c(1,3,4,1)z_3=0;\  c(1,2,4,2)z_2+c(1,3,4,2)z_3=0,  \\
c(1,2,4,3)z_2+c(1,3,4,3)z_3=0;\  c(1,2,4,4)z_2+c(1,3,4,4)z_3=0.
\end{align*}
which is equivalent to $z_2 w_3 + z_3 w_2 = 0$. Therefore $\dim Z(\mathcal{A})=1$ if and only if $Rank B=\dim \langle \{w_2, w_3\}\rangle=1$. In this case,
\begin{align*}
Z(A)&=\{z=\sum_{k=1}^4 z_k e_k \in A : z_1=z_4=0 \text{ and } c(1,2,4,p)z_2+c(1,3,4,p) z_3 =0\}\\
&=\{z_2 e_2 -\frac{z_2 c(1,2,4,p)}{c(1,3,4,p)} e_3 : z_2 \in \mathbb{K}\} \\
&=\{z_2( c(1,3,4,p) e_2 - c(1,2,4,p) e_3) : z_2 \in \mathbb{K}\}
\end{align*}
 if there exists $1\leq p \leq 4$ such that $c(1,3,4,p)\neq 0$, and
\begin{align*}
Z(A)&=\{z=\sum_{k=1}^4 z_k e_k \in A : z_1=z_4=0 \text{ and } c(1,2,4,p)z_2+c(1,3,4,p) z_3 =0\}\\
&=\{-z_3 \frac{c(1,3,4,p)}{c(1,2,4,p)}e_2 + z_3 e_3: z_3 \in \mathbb{K}\}\\
&=\{z_3( c(1,3,4,p) e_2 - c(1,2,4,p)e_3): z_3 \in \mathbb{K}\}\\
&=\{z_3 e_3: z_3 \in \mathbb{K}\}
\end{align*}
 otherwise.
By Proposition \ref{centernil_n+1}, $\mathcal{A}$ is nilpotent if and only if $Z(\mathcal{A})=\bracket{A,A,A}$, since $\dim Z(\mathcal{A})=1$. Now, we prove that this is equivalent to
\begin{align}
c(1,2,4,1)=c(1,2,4,4)=c(1,3,4,1)=c(1,3,4,4)&=0, \nonumber \\
c(1,3,4,p)c(1,2,4,3)+c(1,2,4,p)c(1,2,4,2)&=0, \label{sysnil} \\
c(1,3,4,p)c(1,3,4,3)+c(1,2,4,p)c(1,3,4,2)&=0.\nonumber
\end{align}
$Z(\mathcal{A})=[A,A,A]$ if and only if $\dim \langle\{ w_2, w_3, c(1,3,4,p) e_2 - c(1,2,4,p)e_3\} \rangle =1$, which is equivalent to
$Rank \begin{pmatrix}
 c(1,3,4,1) & -c(1,2,4,1) & 0 \\
 c(1,3,4,2) & -c(1,2,4,2) & c(1,3,4,p) \\
 c(1,3,4,3) & -c(1,2,4,3) & - c(1,2,4,p) \\
 c(1,3,4,4) & -c(1,2,4,4) & 0 \\
\end{pmatrix}
=1$, that is all the $2\times 2$ minors of this matrix are zero, which gives the system \eqref{sysnil}.
\qed
\end{proof}

\begin{corollary}\label{subclasses1}
The class of $3$-Hom-Lie algebras $4_{3,N(2),6}$ with $B\neq 0$ can be split into five non-isomorphic subclasses:
\begin{enumerate}[label=\textup{\arabic*)},ref=\textup{\arabic*},leftmargin=*]
\item $3$-solvable of class $2$, non-$2$-solvable, non-nilpotent, with trivial center:
\begin{align*}
\bracket{e_1,e_2,e_3}&=0\\
\bracket{e_1,e_2,e_4}&=c(1,2,4,1) e_1 + c(1,2,4,2) e_2 + c(1,2,4,3) e_3 + c(1,2,4,4) e_4\\
\bracket{e_1,e_3,e_4}&=c(1,3,4,1) e_1 + c(1,3,4,2) e_2 + c(1,3,4,3) e_3 + c(1,3,4,4) e_4\\
\bracket{e_2,e_3,e_4}&=0
\end{align*}
with $d(1,4)\neq 0$, in that case we have $Rank \begin{pmatrix}
d(2,4) & d(3,4)\\
d(1,2) & d(1,3)
\end{pmatrix} =2$.
\item $3$-solvable of class $2$, $2$-solvable of class $3$, non-nilpotent, with trivial center:
\begin{align*}
\bracket{e_1,e_2,e_3}&=0\\
\bracket{e_1,e_2,e_4}&=c(1,2,4,1) e_1 + c(1,2,4,2) e_2 + c(1,2,4,3) e_3 + c(1,2,4,4) e_4\\
\bracket{e_1,e_3,e_4}&=\lambda c(1,2,4,1) e_1 + c(1,3,4,2) e_2 + c(1,3,4,3) e_3 +\lambda c(1,2,4,4) e_4\\
\bracket{e_2,e_3,e_4}&=0
\end{align*}
with $(c(1,2,4,1),c(1,2,4,4))\neq (0,0)$ or
\begin{align*}
\bracket{e_1,e_2,e_3}&=0\\
\bracket{e_1,e_2,e_4}&= c(1,2,4,2) e_2 + c(1,2,4,3) e_3 \\
\bracket{e_1,e_3,e_4}&=c(1,3,4,1) e_1 + c(1,3,4,2) e_2 + c(1,3,4,3) e_3 + c(1,3,4,4) e_4\\
\bracket{e_2,e_3,e_4}&=0
\end{align*}
such that $Rank \begin{pmatrix}
d(2,4) & d(3,4)\\
d(1,2) & d(1,3)
\end{pmatrix} =1$.
\item $3$-solvable of class $2$, $2$-solvable of class $2$, non-nilpotent, with trivial center:
\begin{align*}
\begin{array}{ll}
\bracket{e_1,e_2,e_3}&=0\\
\bracket{e_1,e_2,e_4}&=c(1,2,4,2) e_2 + c(1,2,4,3) e_3\\
\bracket{e_1,e_3,e_4}&=c(1,3,4,2) e_2 + c(1,3,4,3) e_3\\
\bracket{e_2,e_3,e_4}&=0
\end{array} \quad , \quad \text{with} \ d(2,3)\neq 0.
\end{align*}
\item $3$-solvable of class $2$, $2$-solvable of class $2$, non-nilpotent, with $1$-dimensional center:
\begin{align*}
\bracket{e_1,e_2,e_3}&=0\\
\bracket{e_1,e_2,e_4}&=c(1,2,4,1) e_1 + c(1,2,4,2) e_2 + c(1,2,4,3) e_3 + c(1,2,4,4) e_4\\
\bracket{e_1,e_3,e_4}&=\lambda c(1,2,4,1) e_1 + \lambda c(1,2,4,2) e_2 + \lambda c(1,2,4,3) e_3 \\
&\hspace{4cm} + \lambda c(1,2,4,4) e_4 \\
\bracket{e_2,e_3,e_4}&=0
\end{align*}
with $\bracket{e_1,e_2,e_4}\neq 0$ (that is not all $c(1,2,4,1)$, $c(1,2,4,2)$, $c(1,2,4,3)$, $c(1,2,4,4)$ are zero), or
\begin{align*}
\bracket{e_1,e_2,e_3}&=0\\
\bracket{e_1,e_2,e_4}&=0\\
\bracket{e_1,e_3,e_4}&= c(1,3,4,1) e_1 + c(1,3,4,2) e_2 + c(1,3,4,3) e_3 + c(1,3,4,4) e_4\\
\bracket{e_2,e_3,e_4}&=0
\end{align*}
\item $3$-solvable of class $2$, $2$-solvable of class $2$, nilpotent of class $2$, with $1$-dimen\-sional center:
\begin{align*}
&\begin{array}{ll}
\bracket{e_1,e_2,e_3}&=0\\
\bracket{e_1,e_2,e_4}&= c(1,2,4,2) e_2 + c(1,2,4,3) e_3 \\
\bracket{e_1,e_3,e_4}&= \frac{-c(1,2,4,2)^2}{c(1,2,4,3)} e_2 - c(1,2,4,2) e_3 \\
\bracket{e_2,e_3,e_4}&=0\\
\end{array}
\quad , \quad c(1,2,4,3)  \neq 0\\
&\text{or}  \\
& \begin{array}{ll}
\bracket{e_1,e_2,e_3}&=0\\
\bracket{e_1,e_2,e_4}&= c(1,2,4,2) e_2 + \frac{-c(1,2,4,2)^2}{c(1,3,4,2)} e_3 \\
\bracket{e_1,e_3,e_4}&= c(1,3,4,2) e_2 - c(1,2,4,2) e_3 \\
\bracket{e_2,e_3,e_4}&=0\\
\end{array}
\quad , \quad c(1,3,4,2)\neq 0
\end{align*}
\end{enumerate}
\end{corollary}

\begin{remark}
In the last case above, either $c(1,3,4,2)\neq 0$ or $c(1,2,4,3)\neq 0$, if both are zero, then the bracket is zero.
\end{remark}

\begin{corollary}
In the subclasses presented in Corollary \ref{subclasses1}, cases 1 and 3 cannot be multiplicative. All the multiplicative $3$-Hom-Lie algebras in the considered class are contained in the remaining subclasses:
\begin{enumerate}
\item[\textup{2m)}] $3$-solvable of class $2$, $2$-solvable of class $3$, non-nilpotent, with trivial center
\begin{align*}
\bracket{e_1,e_2,e_3}&=0\\
\bracket{e_1,e_2,e_4}&=c(1,2,4,1) e_1 + c(1,2,4,2) e_2\\
\bracket{e_1,e_3,e_4}&=c(1,3,4,1) e_1 + c(1,3,4,2) e_2 \\
\bracket{e_2,e_3,e_4}&=0
\end{align*}
with $d(1,2)=c(1,2,4,1)c(1,3,4,2)-c(1,2,4,2)c(1,3,4,1)\neq 0$.
\item[\textup{4m)}] $3$-solvable of class $2$, $2$-solvable of class $2$, non-nilpotent, with $1$-dimensional center
\begin{align*}
\bracket{e_1,e_2,e_3}&=0\\
\bracket{e_1,e_2,e_4}&=c(1,2,4,1) e_1 + c(1,2,4,2) e_2\\
\bracket{e_1,e_3,e_4}&=\lambda c(1,2,4,1) e_1 + \lambda  c(1,2,4,2) e_2\\
\bracket{e_2,e_3,e_4}&=0
\end{align*}
\item[\textup{5m)}] $3$-solvable of class $2$, $2$-solvable of class $2$, nilpotent of class $2$, with $1$-dimensional center
$$\begin{array}{lll}
\bracket{e_1,e_2,e_3}&=&0\\
\bracket{e_1,e_2,e_4}&=&0 \\
\bracket{e_1,e_3,e_4}&=&c(1,3,4,2) e_2 \\
\bracket{e_2,e_3,e_4}&=&0
\end{array}\quad ,\quad c(1,3,4,2)\neq 0.
$$
\end{enumerate}
\end{corollary}

\section{Isomorphism classes for \texorpdfstring{$4_{3,N(2),6}$}{}} \label{sec:IsomorphismClasses}
The following theorem gives the classification up to isomorphism of the class of $3$-Hom-Lie  algebras $4_{3,N(2),6}$. Note that isomorphisms are considered in the sense of Hom-algebras, that is they are required to intertwine not only the multiplications, but also the twisting maps.

\begin{theorem}
Any $3$-Hom-Lie algebra $\mathcal{A}$ in the class of $3$-Hom-Lie algebras $4_{3,N(2),6}$ with $B\neq 0$ is isomorphic to one of the following:
\begin{enumerate}[label=\textup{\arabic*)},ref=\textup{\arabic*},leftmargin=*]
\item $\dim D^1_3(\mathcal{A})=2$, non-$2$-solvable, non-nilpotent, with trivial center:
	\begin{enumerate}[leftmargin=*]
	\item $c(1,2,4,4)\neq 0$.
			\begin{align*}
			\bracket{e_1,e_2,e_3}&=0\\
			\bracket{e_1,e_2,e_4}&= e_4 \\
			\bracket{e_1,e_3,e_4}&= c'(1,3,4,1) e_1 + c'(1,3,4,3) e_3 + c'(1,3,4,4) e_4\\
			\bracket{e_2,e_3,e_4}&=0,
			\\[1mm] 			
            c'(1,3,4,1)&=\frac{-d(1,4)}{c(1,2,4,4)}\neq 0,\quad c'(1,3,4,3)= \frac{-d(3,4)}{c(1,2,4,4)^2},\\
			c'(1,3,4,4)&=\frac{c(1,2,4,3)+c(1,3,4,4)}{c(1,2,4,4)}.
			\end{align*}
			 Two such algebras, given by the structure constants $(c'(i,j,k,p))$ and $(c''(i,j,k,p))$ respectively, are isomorphic if and only if $c'(1,3,4,3)=c''(1,3,4,3)$, $c'(1,3,4,4)=c''(1,3,4,4)$ and $\frac{c'(1,3,4,1)}{c''(1,3,4,1)}$ is a square in $\mathbb{K}$. Thus, this family of algebras up to isomorphism is parametrized by  $\frac{\mathbb{K}^*}{(\mathbb{K}^*)^2}\times \mathbb{K}\times \mathbb{K}$, where$\frac{\mathbb{K}^*}{(\mathbb{K}^*)^2}$ is the factor group of $\mathbb{K}^*$ by $(\mathbb{K}^*)^2=\{x^2 | x \in \mathbb{K}^*\}$.
	\item $c(1,2,4,4)= 0$, $c(1,2,4,3)\neq 0$ and $c(1,2,4,3)\neq c(1,3,4,4)$. In this case $c(1,2,4,1)$ and $c(1,3,4,4)$ are non-zero since $d(1,4)\neq 0$.
			\begin{align*}
\begin{array}{lll}
			\bracket{e_1,e_2,e_3}&=& 0\\
			\bracket{e_1,e_2,e_4}&=& c'(1,2,4,1) e_1 + e_3\\
			\bracket{e_1,e_3,e_4}&=& c'(1,3,4,4) e_4\\
			\bracket{e_2,e_3,e_4}&=& 0,
\end{array} \quad , \quad
\begin{array}[c]{ll}
c'(1,2,4,1)&=c(1,2,4,1)\neq 0,\\
c'(1,3,4,4)&=\frac{c(1,3,4,4)}{c(1,2,4,3)}\neq 0.
\end{array}
\end{align*}
			Two such algebras, given by the structure constants $(c'(i,j,k,p))$ and $(c''(i,j,k,p))$ respectively, are isomorphic if and only if $\frac{c'(1,2,4,1)}{c''(1,2,4,1)}$ is a square in $\mathbb{K}$.
	\item  $c(1,2,4,4)= 0$, $c(1,2,4,3)\neq 0$ and $c(1,2,4,3) = c(1,3,4,4)$. In this case also $c(1,2,4,1)$ and $c(1,3,4,4)$ are non-zero since $d(1,4)\neq 0$.
			\begin{align*}    \begin{array}{ll}
			\bracket{e_1,e_2,e_3}&=0\\
			\bracket{e_1,e_2,e_4}&= c'(1,3,4,3) e_3 + e_4 \\
			\bracket{e_1,e_3,e_4}&= c'(1,2,4,1) e_1 + e_3\\
			\bracket{e_2,e_3,e_4}&=0,\end{array} \quad, \quad
\begin{array}[c]{ll} 			
c'(1,3,4,3)&=\frac{c(1,3,4,4)}{c(1,2,4,3)}\neq 0,\\
c'(1,2,4,1)&=c(1,2,4,1)\neq 0.
\end{array}
			\end{align*}
			Two such algebras, given by the structure constants $(c'(i,j,k,p))$ and $(c''(i,j,k,p))$ respectively are isomorphic if and only if $\frac{c'(1,2,4,1)}{c''(1,2,4,1)}$ is a square in $\mathbb{K}$.
	\item $c(1,2,4,4)= 0$ and $c(1,2,4,3)= 0$. Similarly, in this case $c(1,2,4,1)$ and $c(1,3,4,4)$ are non-zero since $d(1,4)\neq 0$. 	
			\begin{align*}
            \begin{array}{ll}
			\bracket{e_1,e_2,e_3}&=0\\
			\bracket{e_1,e_2,e_4}&= c'(1,2,4,1) e_1 \\
			\bracket{e_1,e_3,e_4}&= e_4\\
			\bracket{e_2,e_3,e_4}&=0,
            \end{array} \quad, \quad
			c'(1,2,4,1)=c(1,2,4,1). 			
            \end{align*}
	Two such brackets given by the structure constants $(c'(i,j,k,p))$ and $(c''(i,j,k,p))$ are isomorphic if and only if $\frac{c'(1,2,4,1)}{c''(1,2,4,1)}$ is a square in $\mathbb{K}$. In particular, If $c(1,2,4,1)$ is a square in $\mathbb{K}$, we get the following bracket
			\begin{align*}
			\bracket{e_1,e_2,e_3}&=0\\
			\bracket{e_1,e_2,e_4}&= e_1 \\
			\bracket{e_1,e_3,e_4}&= e_4\\
			\bracket{e_2,e_3,e_4}&=0.
			\end{align*}
	\end{enumerate}
\item $\dim D^1_3(\mathcal{A})=2$, $2$-solvable of class $3$, non-nilpotent, with trivial center, that is $d(1,4)=0$, $\begin{pmatrix}
d(2,4) & d(3,4)\\
d(1,2) & d(1,3)
\end{pmatrix}\neq0$, equivalent to $(c(1,2,4,1),c(1,2,4,4)) \neq (0,0)$ and $(c(1,3,4,1),c(1,3,4,4))= \lambda (c(1,2,4,1),c(1,2,4,4))$ for some $\lambda \in \mathbb{K}$ or $(c(1,2,4,1),c(1,2,4,4)) = (0,0)$ and $(c(1,3,4,1),c(1,3,4,4)) \neq (0,0)$:
	\begin{enumerate}
	\item$c(1,2,4,4) \neq 0$, hence $(c(1,2,4,1),c(1,2,4,4)) \neq (0,0)$
			\begin{align*}
			\bracket{e_1,e_2,e_3}&=0\\
			\bracket{e_1,e_2,e_4}&= e_4 \\
			\bracket{e_1,e_3,e_4}&= c'(1,3,4,2) e_2 + c'(1,3,4,3) e_3 + c'(1,3,4,4) e_4\\
			\bracket{e_2,e_3,e_4}&=0,
			\\
			c'(1,3,4,2)&=\frac{\lambda c(1,2,4,3)^2-\lambda c(1,2,4,2) c(1,2,4,4)}{c(1,2,4,4)^2}\\
			&\quad+\frac{-c(1,3,4,3) c(1,2,4,3)+c(1,2,4,4) c(1,3,4,2)}{c(1,2,4,4)^2},\\
			c'(1,3,4,3)&=\frac{c(1,3,4,3)-\lambda c(1,2,4,3)}{c(1,2,4,4)},\\
			c'(1,3,4,4)&=\frac{\lambda c(1,2,4,4)+c(1,2,4,3)}{c(1,2,4,4)}.
			\end{align*}
			Any two different brackets of this form give non-isomorphic $3$-Hom-Lie algebras.
	\item $(c(1,2,4,1),c(1,2,4,4)) \neq (0,0)$ and $c(1,2,4,4) = 0$, which means that $c(1,2,4,1)\neq 0$ (else the algebra would be $2$-solvable of class $2$). For $c(1,2,4,3)\neq 0$,
			\begin{align*}
			\bracket{e_1,e_2,e_3}&=0\\
			\bracket{e_1,e_2,e_4}&= c'(1,2,4,1) e_1 + e_3 \\
			\bracket{e_1,e_3,e_4}&= \lambda' c'(1,2,4,1) e_1+ c'(1,3,4,2) e_2\\
			\bracket{e_2,e_3,e_4}&=0,
			\\ 			
            c'(1,2,4,1)&=c(1,2,4,1)\neq 0,\quad
            \lambda' =\frac{\lambda c(1,2,4,3)-c(1,3,4,3)}{c(1,2,4,3)}, \\
            c'(1,3,4,2) &=\frac{-\lambda c(1,2,4,3) c(1,3,4,3)-\lambda c(1,2,4,2) c(1,2,4,3)}{c(1,2,4,3)^2}\\
			&\hspace{2cm} +\frac{c(1,3,4,3)^2+c(1,2,4,3) c(1,3,4,2)}{c(1,2,4,3)^2}.
			\end{align*}
			 Two such brackets given by the structure constants $(c'(i,j,k,p))$ and $(c''(i,j,k,p))$ define isomorphic algebras if and only if $\frac{c'(1,2,4,1)}{c''(1,2,4,1)}$ is a square in $\mathbb{K}$.
	\item $(c(1,2,4,1),c(1,2,4,4)) \neq (0,0)$ and $c(1,2,4,4) = 0$, which means that $c(1,2,4,1)\neq 0$ (else the algebra would be $2$-solvable of class $2$). For $c(1,2,4,3)= 0$ and $c(1,3,4,3)\neq 0$,
			\begin{eqnarray*}
         &&   \begin{array}{l}
			 \bracket{e_1,e_2,e_3}=0\\
			 \bracket{e_1,e_2,e_4}= c'(1,2,4,1) e_1 \\
			 \bracket{e_1,e_3,e_4}= \lambda' c'(1,2,4,1) e_1+e_3\\
			 \bracket{e_2,e_3,e_4}=0,
			\end{array} \\
\lambda' &=&\frac{-\lambda c(1,2,4,2)+\lambda c(1,3,4,3)+c(1,3,4,2)}{c(1,3,4,3)},\\
&& c'(1,2,4,1) = c(1,2,4,1).
			\end{eqnarray*}
			Two such brackets given by the structure constants $(c'(i,j,k,p))$ and $(c''(i,j,k,p))$ are isomorphic if and only if $\frac{c'(1,2,4,1)}{c''(1,2,4,1)}$ is a square in $\mathbb{K}$.
	\item $(c(1,2,4,1),c(1,2,4,4)) \neq (0,0)$ and $c(1,2,4,4) = 0$, which means that $c(1,2,4,1)\neq 0$ (else the algebra would be $2$-solvable of class $2$). We consider $c(1,2,4,3)= 0$ and $c(1,3,4,3)=0$.  In this case, the algebra is multiplicative.
		    \begin{align*}
            \begin{array}{ll}
			\bracket{e_1,e_2,e_3}&=0\\
			\bracket{e_1,e_2,e_4}&= c'(1,2,4,1) e_1 \\
			\bracket{e_1,e_3,e_4}&= e_2\\
			\bracket{e_2,e_3,e_4}&=0,
            \end{array} \quad, \quad
			c'(1,2,4,1)&=c(1,2,4,1)\neq 0.
			\end{align*}
			Two such brackets given by the structure constants $(c'(i,j,k,p))$ and $(c''(i,j,k,p))$ define isomorphic algebras if and only if $\frac{c'(1,2,4,1)}{c''(1,2,4,1)}$ is a square in $\mathbb{K}$.
			\item  $(c(1,2,4,1),c(1,2,4,4)) = (0,0)$ and $(c(1,3,4,1),c(1,3,4,4)) \neq (0,0)$, $c(1,2,4,3)\neq 0$
			\begin{align*}
			\bracket{e_1,e_2,e_3}&=0\\
			\bracket{e_1,e_2,e_4}&= e_3 \\
			\bracket{e_1,e_3,e_4}&= c'(1,3,4,3) e_3 + c'(1,3,4,4) e_4\\
			\bracket{e_2,e_3,e_4}&=0, \\
			c'(1,3,4,3) &=\frac{c(1,2,4,2) c(1,2,4,3)+c(1,3,4,3) c(1,2,4,3)}{c(1,2,4,3)^2} \\
			&\quad - \frac{c(1,2,4,2) c(1,3,4,4)}{c(1,2,4,3)^2},\\
			c'(1,3,4,4)&=\frac{c(1,3,4,4)}{c(1,2,4,3)}.
			\end{align*}
			Any two different brackets of this form give non-isomorphic $3$-Hom-Lie algebras.
			\item $(c(1,2,4,1),c(1,2,4,4)) = (0,0)$ and $(c(1,3,4,1),c(1,3,4,4)) \neq (0,0)$, $c(1,2,4,3)=0$, $c(1,3,4,4)\neq 0$
			\begin{align*}
            \begin{array}{ll}
			\bracket{e_1,e_2,e_3}&=0\\
			\bracket{e_1,e_2,e_4}&= c'(1,2,4,2) e_2 \\
			\bracket{e_1,e_3,e_4}&= e_4\\
			\bracket{e_2,e_3,e_4}&=0,
            \end{array} \quad, \quad
			c'(1,2,4,2)&=\frac{c(1,2,4,2)}{c(1,3,4,4)}.
			\end{align*}
			\item $(c(1,2,4,1),c(1,2,4,4)) = (0,0)$ and $(c(1,3,4,1),c(1,3,4,4)) \neq (0,0)$, $c(1,2,4,3)=0$, $c(1,3,4,4)= 0$ and $c(1,3,4,3)\neq 0$
			\begin{align*}
            \begin{array}{ll}
			\bracket{e_1,e_2,e_3}&=0\\
			\bracket{e_1,e_2,e_4}&= c'(1,2,4,2) e_2 \\
			\bracket{e_1,e_3,e_4}&= c'(1,3,4,1) e_1 + e_3\\
			\bracket{e_2,e_3,e_4}&=0,
            \end{array} \quad, \quad
            \begin{array}{ll}
			c'(1,2,4,2)&=\frac{c(1,2,4,2)}{c(1,3,4,3)},\\
			c'(1,3,4,1)&=c(1,3,4,1).
			\end{array}
            \end{align*}
			Two such brackets given by the structure constants $(c'(i,j,k,p))$ and $(c''(i,j,k,p))$ define isomorphic algebras if and only if $c'(1,2,4,2)=c''(1,2,4,2)$ and $\frac{c'(1,3,4,1)}{c''(1,3,4,1)}$ is a square in $\mathbb{K}$.
			\item $(c(1,2,4,1),c(1,2,4,4)) = (0,0)$ and $(c(1,3,4,1),c(1,3,4,4)) \neq (0,0)$, $c(1,2,4,3)=0$, $c(1,3,4,4)= 0$ and $c(1,3,4,3)=0$. This algebra is multiplicative,
			\begin{align*}
             \begin{array}{ll}
			\bracket{e_1,e_2,e_3}&=0\\
			\bracket{e_1,e_2,e_4}&= e_2 \\
			\bracket{e_1,e_3,e_4}&= c'(1,3,4,1) e_1\\
			\bracket{e_2,e_3,e_4}&=0
            \end{array} \quad, \quad			
            c'(1,3,4,1)=c(1,3,4,1)\neq 0.
			\end{align*}
			Two such brackets given by the structure constants $(c'(i,j,k,p))$ and $(c''(i,j,k,p))$ define isomorphic algebras if and only if $\frac{c'(1,3,4,1)}{c''(1,3,4,1)}$ is a square in $\mathbb{K}$.
	\end{enumerate}
\item  $\dim D^1_3(\mathcal{A})=2$, $2$-solvable of class $2$, non-nilpotent, with trivial center:
	\begin{enumerate}
	\item $c(1,2,4,3)\neq 0$
		     \begin{align*}
			\bracket{e_1,e_2,e_3}&=0\\
			\bracket{e_1,e_2,e_4}&= c'(1,2,4,2) e_2 + e_3 \\
			\bracket{e_1,e_3,e_4}&= c'(1,3,4,2) e_2\\
			\bracket{e_2,e_3,e_4}&=0,
			\end{align*}
			\begin{align*}
			 c'(1,3,4,2)&=\frac{c(1,2,4,3) c(1,3,4,2)-c(1,2,4,2) c(1,3,4,3)}{c(1,2,4,3)^2},\\
			 c'(1,2,4,2)&=\frac{-c(1,2,4,2)-c(1,3,4,3)}{c(1,2,4,3)}.
			 \end{align*}
	\item $c(1,2,4,3)=0$, $c(1,3,4,3)\neq 0$ and $c(1,3,4,3)\neq c(1,2,4,2)$
		    \begin{align*}
            \begin{array}{ll}
			\bracket{e_1,e_2,e_3}&=0\\
			\bracket{e_1,e_2,e_4}&= c'(1,2,4,2) e_2  \\
			\bracket{e_1,e_3,e_4}&= e_3\\
			\bracket{e_2,e_3,e_4}&=0
            \end{array} \quad, \quad
			c'(1,2,4,2)&=\frac{c(1,2,4,2)}{c(1,3,4,3)}.
			\end{align*}
	\item $c(1,2,4,3)=0$, $c(1,3,4,3)\neq 0$ and $c(1,3,4,3)=c(1,2,4,2)$
		    \begin{align*}
            \begin{array}{ll}
			\bracket{e_1,e_2,e_3}&=0\\
			\bracket{e_1,e_2,e_4}&= e_2 \\
			\bracket{e_1,e_3,e_4}&= c'(1,3,4,2) e_2 + e_3\\
			\bracket{e_2,e_3,e_4}&=0
            \end{array} \quad, \quad
			c'(1,3,4,2)&=\frac{c(1,3,4,2)}{c(1,3,4,3)}.
			\end{align*}
	\end{enumerate}
\item $\dim D^1_3(\mathcal{A})=1$, $2$-solvable of class $2$, non-nilpotent, with $1$-dimensional center:
	\begin{enumerate}
	\item $w_3\neq 0$, $w_2=\lambda w_3$ with $\lambda \in \mathbb{K}$ and $c(1,2,4,4)\neq 0$
			\begin{equation*}
			\begin{array}{lll}
            \bracket{e_1,e_2,e_3}&=& 0\\
		    \bracket{e_1,e_2,e_4}&=& e_4 \\
			\bracket{e_1,e_3,e_4}&=& \lambda' e_4\\
			\bracket{e_2,e_3,e_4}&=0
            \end{array}\quad , \quad
\lambda'=\frac{c(1,2,4,3)}{c(1,2,4,4)}+\lambda.
			\end{equation*}
			Two such brackets with parameters $\lambda'$ and $\lambda''$ define isomorphic $3$-Hom-Lie algebras if and only if $\lambda'=\lambda''.$
	\item $w_3\neq 0$, $w_2=\lambda w_3$ with $\lambda \in \mathbb{K}$ and $c(1,2,4,4)= 0$, $c(1,2,4,3)\neq 0$, $c(1,2,4,1)\neq 0$
			\begin{align*}
            \begin{array}{ll}
			\bracket{e_1,e_2,e_3}&=0\\
			\bracket{e_1,e_2,e_4}&= e_1+c'(1,2,4,3)e_3 \\
			\bracket{e_1,e_3,e_4}&= 0\\
			\bracket{e_2,e_3,e_4}&=0
            \end{array}\quad , \quad
			c'(1,2,4,3)&=c(1,2,4,3)\neq 0.
			\end{align*}
			Two such brackets given by the structure constants $(c'(i,j,k,p))$ and $(c''(i,j,k,p))$ define isomorphic algebras if and only if $\frac{c'(1,2,4,3)}{c''(1,2,4,3)}$ is a square in $\mathbb{K}$.
	\item $w_3\neq 0$, $w_2=\lambda w_3$ with $\lambda \in \mathbb{K}$ and $c(1,2,4,4)= 0$, $c(1,2,4,3)\neq 0$, $c(1,2,4,1)=0$
			\begin{align*}
			\bracket{e_1,e_2,e_3}&=0\\
			\bracket{e_1,e_2,e_4}&= e_3 \\
			\bracket{e_1,e_3,e_4}&= \lambda' e_3\\
			\bracket{e_2,e_3,e_4}&=0,
			\end{align*}
	\item $w_3\neq 0$, $w_2=\lambda w_3$ with $\lambda \in \mathbb{K}$ and $c(1,2,4,4)= 0$, $c(1,2,4,3)= 0$, $c(1,2,4,1)\neq 0$
			\begin{align*}
			\bracket{e_1,e_2,e_3}&=0\\
			\bracket{e_1,e_2,e_4}&= c'(1,2,4,1)e_1 \\
			\bracket{e_1,e_3,e_4}&= 0\\
			\bracket{e_2,e_3,e_4}&=0,
			\end{align*}
	\item $w_3\neq 0$, $w_2=\lambda w_3$ with $\lambda \in \mathbb{K}$ and $c(1,2,4,4)= 0$, $c(1,2,4,3)= 0$, $c(1,2,4,1)= 0$, $c(1,2,4,2)\neq 0$
			\begin{align*}
			\bracket{e_1,e_2,e_3}&=0\\
			\bracket{e_1,e_2,e_4}&= e_2 \\
			\bracket{e_1,e_3,e_4}&= 0\\
			\bracket{e_2,e_3,e_4}&=0,
			\end{align*}
	\item $w_3=0$, $c(1,3,4,4)\neq 0$
			\begin{align*}
			\bracket{e_1,e_2,e_3}&=0\\
			\bracket{e_1,e_2,e_4}&= 0 \\
			\bracket{e_1,e_3,e_4}&= e_4\\
			\bracket{e_2,e_3,e_4}&=0,
			\end{align*}
	\item $w_3=0$, $c(1,3,4,4) = 0$, $c(1,3,4,1) \neq 0$, $c(1,3,4,3) \neq 0$
			\begin{align*}
            \begin{array}{ll}
			\bracket{e_1,e_2,e_3}&=0\\
			\bracket{e_1,e_2,e_4}&= 0 \\
			\bracket{e_1,e_3,e_4}&= c'(1,3,4,1)e_1 + e_3\\
			\bracket{e_2,e_3,e_4}&=0
            \end{array}\quad , \quad
			c'(1,3,4,1)&=c(1,3,4,1).
			\end{align*}
			Two such brackets given by the structure constants $(c'(i,j,k,p))$ and $(c''(i,j,k,p))$ define isomorphic algebras if and only if $\frac{c'(1,3,4,1)}{c''(1,3,4,1)}$ is a square in $\mathbb{K}$.
	\item $w_3=0$, $c(1,3,4,4) = 0$, $c(1,3,4,1) \neq 0$, $c(1,3,4,3) = 0$
			\begin{align*}
            \begin{array}{ll}
			\bracket{e_1,e_2,e_3}&=0\\
			\bracket{e_1,e_2,e_4}&= \\
			\bracket{e_1,e_3,e_4}&= c'(1,3,4,1) e_1\\
			\bracket{e_2,e_3,e_4}&=0
            \end{array}\quad , \quad
			c'(1,3,4,1)&=c(1,3,4,1)
			\end{align*}
			Two such brackets given by the structure constants $(c'(i,j,k,p))$ and $(c''(i,j,k,p))$ define isomorphic algebras if and only if $\frac{c'(1,3,4,1)}{c''(1,3,4,1)}$ is a square in $\mathbb{K}$. This bracket defines a multiplicative algebra.
	\end{enumerate}
\item $\dim D^1_3(\mathcal{A})=1$, $2$-solvable of class $2$, nilpotent of class $2$, with $1$-dimensional center:

\begin{enumerate*}[series=MyListThmclf43N26case5a5b,before=\hspace{-0.6ex}]
	\item $\begin{array}[t]{rl}
			\bracket{e_1,e_2,e_3}&=0\\
			\bracket{e_1,e_2,e_4}&=0 \\
			\bracket{e_1,e_3,e_4}&=e_2\\
			\bracket{e_2,e_3,e_4}&=0
			\end{array}$  \hspace{1cm}
	\item $\begin{array}[t]{rl}
			\bracket{e_1,e_2,e_3}&=0\\
			\bracket{e_1,e_2,e_4}&=e_3 \\
			\bracket{e_1,e_3,e_4}&=0\\
			\bracket{e_2,e_3,e_4}&=0
			\end{array}$
	\end{enumerate*}
\end{enumerate}

\end{theorem}

\begin{proof}
Let $\mathcal{A}=(A,\bracket{\cdot,\dots,\cdot},\alpha)$ be a $3$-Hom-Lie algebra in one of the classes presented in Corollary \ref{subclasses1} and consider the matrix $B$ defining its bracket in a basis $(e_i)$ where $\alpha$ is in its Jordan normal form. Any $3$-Hom-Lie algebra isomorphic to $\mathcal{A}$ has its bracket given by a matrix $B'=\frac{1}{\det (P)} P B P^T$ where $P$ is an invertible matrix that commutes with $[\alpha]$, the matrix representing $\alpha$ in the basis $(e_i)$. A matrix $P=(p(i,j))_{1\leq i,j\leq 4}$ commutes with $[\alpha]=
\begin{pmatrix}
 0 & 0 & 0 & 0 \\
 0 & 0 & 1 & 0 \\
 0 & 0 & 0 & 1 \\
 0 & 0 & 0 & 0
\end{pmatrix}$ if and only if it is of the form
$
 P=\begin{pmatrix}
	p(1,1) & 0 & 0 & p(1,4) \\
 	p(2,1) & p(3,3) & p(2,3) & p(2,4) \\
 	0 & 0 & p(3,3) & p(2,3) \\
 	0 & 0 & 0 & p(3,3) \\
 	\end{pmatrix},
$
with $\det(P)\neq 0$ that is $p(1,1)p(3,3)^3\neq 0$ which is equivalent to $p(1,1)\neq 0$ and $p(3,3)\neq 0$. We denote by $c'(i,j,k,p)$ the structure constants of the bracket after the transformation by $P$.

In the following, in the matrix $B'$, there appear structure constants of the form $c'(i,j,k,l)=\frac{c(i,j,k,l)}{p(1,1) p(3,3)}$ or $\frac{c(i,j,k,l)}{p(3,3)^2}$. Note that, since $p(1,1)\neq 0$ and $p(3,3)\neq 0$,
\begin{equation}\label{nonisom}
\frac{c(i,j,k,l)}{p(1,1) p(3,3)}=0 \text{ or } \frac{c(i,j,k,l)}{p(3,3)^2}=0 \iff c(i,j,k,l)=0,
\end{equation}
and thus in such a case, the algebras given by the bracket with $c(i,j,k,l)=0$ and the bracket with $c(i,j,k,l)\neq 0$ cannot be isomorphic.
\begin{enumerate}[leftmargin=*,wide,labelwidth=!,labelindent=0pt]
\item $\dim D^1_3(\mathcal{A})=2$, non-$2$-solvable, non-nilpotent, with trivial center, that is
\[
B= \begin{pmatrix}
 0 & c(1,3,4,1) & -c(1,2,4,1) & 0 \\
 0 & c(1,3,4,2) & -c(1,2,4,2) & 0 \\
 0 & c(1,3,4,3) & -c(1,2,4,3) & 0 \\
 0 & c(1,3,4,4) & -c(1,2,4,4) & 0
\end{pmatrix},
\]
with $d(1,4)=c(1,2,4,1) c(1,3,4,4) - c(1,2,4,4) c(1,3,4,1)\neq 0$.
\begin{align*}
&B'=\frac{1}{\det (P)} P B P^T=\\
&\left(\begin{array}{cccc}
 0 & b'(1,2) & \frac{-c(1,2,4,1) p(1,1)-c(1,2,4,4) p(1,4)}{p(1,1) p(3,3)^2} &
   0 \\%
 0 & b'(2,2)  & b'(2,3) & 0 \\%
 0 & b'(3,2) & \frac{-c(1,2,4,4) p(2,3)-c(1,2,4,3) p(3,3)}{p(1,1) p(3,3)^2} &
   0 \\%
 0 & \frac{c(1,3,4,4) p(3,3)^2-c(1,2,4,4) p(2,3) p(3,3)}{p(1,1) p(3,3)^3} &
   -\frac{c(1,2,4,4)}{p(1,1) p(3,3)} & 0 %
\end{array}\right),
\end{align*}
{\scriptsize \begin{align*}
b'(1,2)&=c'(1,3,4,1)=\frac{p(2,3) (-c(1,2,4,1) p(1,1)-c(1,2,4,4) p(1,4)) }{p(1,1) p(3,3)^3} \\  &\hspace{2cm} +\frac{p(3,3) (c(1,3,4,1) p(1,1)+c(1,3,4,4)
   p(1,4)) }{p(1,1) p(3,3)^3},\\
b'(2,2)&=c'(1,3,4,2)
\\&=\frac{p(2,3) (-c(1,2,4,1) p(2,1)-c(1,2,4,3) p(2,3)-c(1,2,4,4) p(2,4) -c(1,2,4,2)
   p(3,3))}{p(1,1) p(3,3)^3}\\ &\qquad+\frac{p(3,3) (c(1,3,4,1) p(2,1)+c(1,3,4,3) p(2,3)+c(1,3,4,4) p(2,4)+c(1,3,4,2) p(3,3))}{p(1,1) p(3,3)^3}, \\
b'(2,3)&=-c'(1,2,4,2)=\frac{-c(1,2,4,1) p(2,1)-c(1,2,4,3) p(2,3)}{p(1,1) p(3,3)^2}\\ & \hspace{2.1cm}+\frac{-c(1,2,4,4) p(2,4)-c(1,2,4,2) p(3,3)}{p(1,1) p(3,3)^2},\\
b'(3,2)&=c'(1,3,4,3)= \frac{ p(2,3) (-c(1,2,4,4) p(2,3)-c(1,2,4,3) p(3,3))}{p(1,1) p(3,3)^3 }\\ & \hspace{2cm} + \frac{p(3,3) (c(1,3,4,4) p(2,3)+c(1,3,4,3) p(3,3))}{p(1,1) p(3,3)^3 }.
\end{align*}  }
and notice that $\frac{c(1,2,4,4)}{p(1,1) p(3,3)}=0$ if and only if $c(1,2,4,4)=0$, therefore a bracket with $c(1,2,4,4)=0$ and a bracket with $c(1,2,4,4)\neq 0$ cannot define isomorphic $3$-Hom-Lie algebras.
If $c(1,2,4,4)\neq 0$, then 
choosing
\begin{align*}
&P=P_{1,1}=\\&\begin{pmatrix}
 \frac{c(1,2,4,4)}{p(3,3)} & 0 & 0 &
   -\frac{c(1,2,4,1)}{p(3,3)} \\
 p(2,1) & p(3,3) & -\frac{c(1,2,4,3)
   p(3,3)}{c(1,2,4,4)} & p(2,4) \\
 0 & 0 & p(3,3) & -\frac{c(1,2,4,3) p(3,3)}{c(1,2,4,4)}
   \\
 0 & 0 & 0 & p(3,3)
\end{pmatrix}
\end{align*}
{\scriptsize \begin{align*} p(2,1)&=-\frac{(-c(1,2,4,4)
   c(1,3,4,3) c(1,2,4,3) +c(1,3,4,4) c(1,2,4,3)^2)p(3,3) }{-c(1,2,4,4) d(1,4)} \\  & \qquad + \frac{(c(1,2,4,4)^2
   c(1,3,4,2) -c(1,2,4,2) c(1,2,4,4) c(1,3,4,4))
   p(3,3)}{-c(1,2,4,4) d(1,4)}\\
   p(2,4)&= \frac{ \big(c(1,2,4,1)
   c(1,2,4,4) c(1,3,4,2)-c(1,2,4,1) c(1,3,4,3)
   c(1,2,4,3)\big) p(3,3) }{-c(1,2,4,4)
   d(1,4)}\\ &\qquad +
   \frac{ \big(-c(1,2,4,2) c(1,2,4,4) c(1,3,4,1)+c(1,3,4,1)
   c(1,2,4,3)^2\big) p(3,3) }{-c(1,2,4,4)
   d(1,4)}
   \end{align*}  }
we get $B'=\left(
\begin{array}{cccc}
 0 & \frac{c(1,2,4,4) c(1,3,4,1)-c(1,2,4,1) c(1,3,4,4)}{c(1,2,4,4) p(3,3)^2} & 0 & 0 \\
 0 & 0 & 0 & 0 \\
 0 & \frac{c(1,2,4,4) c(1,3,4,3)-c(1,2,4,3) c(1,3,4,4)}{c(1,2,4,4)^2} & 0 & 0 \\
 0 & \frac{c(1,2,4,3)+c(1,3,4,4)}{c(1,2,4,4)} & -1 & 0
\end{array}
\right).$
If $c(1,2,4,4)=0$, then 
\[
B'=\begin{pmatrix}
 0 & b'(1,2) & -\frac{c(1,2,4,1)}{p(3,3)^2} & 0 \\
 0 & b'(2,2) & b'(2,3)& 0 \\
 0 & b'(3,2) & -\frac{c(1,2,4,3)}{p(1,1) p(3,3)} & 0 \\
 0 & \frac{c(1,3,4,4)}{p(1,1) p(3,3)} & 0 & 0
\end{pmatrix}
\]
{\scriptsize \begin{align*}
& b'(1,2)=c'(1,3,4,1)= \frac{p(3,3) (c(1,3,4,1) p(1,1)+c(1,3,4,4) p(1,4))-c(1,2,4,1) p(1,1) p(2,3)}{p(1,1)
   p(3,3)^3} \\
& b'(2,2)=c'(1,3,4,2)=\frac{p(2,3) (-c(1,2,4,1) p(2,1)-c(1,2,4,3) p(2,3)-c(1,2,4,2) p(3,3))}{p(1,1) p(3,3)^3} \\
& \hspace{2cm} + \frac{p(3,3) (c(1,3,4,1)
   p(2,1)+c(1,3,4,3) p(2,3)+c(1,3,4,4) p(2,4)+c(1,3,4,2) p(3,3))}{p(1,1) p(3,3)^3} \\
& b'(2,3)=-c'(1,2,4,2)= \frac{-c(1,2,4,1) p(2,1)-c(1,2,4,3) p(2,3)-c(1,2,4,2) p(3,3)}{p(1,1) p(3,3)^2} \\
& b'(3,2)=c'(1,3,4,3)= \frac{p(3,3) (c(1,3,4,4) p(2,3)+c(1,3,4,3) p(3,3)) -c(1,2,4,3) p(2,3) p(3,3)}{p(1,1)
   p(3,3)^3}.
\end{align*}}

Using the same argument, consider the cases where each of the structure constants $c(1,3,4,4)$, $c(1,2,4,3)$ and $c(1,2,4,1)$ are zero or non-zero.

If $c(1,3,4,4)=0$, then $d(1,4)=0$ and
\begin{multline*}
{\scriptsize\begin{pmatrix}
d(2,4) & d(3,4)\\
d(1,2) & d(1,3)
\end{pmatrix}}= \\
{\scriptsize \left(\begin{array}{cc}
 0, & 0 \\
 c(1,2,4,1) c(1,3,4,2)-c(1,2,4,2) c(1,3,4,1), & c(1,2,4,1) c(1,3,4,3)-c(1,2,4,3) c(1,3,4,1)
\end{array}\right)}\end{multline*}
has rank less than or equal to $1$, which means that the algebra is $2$-solvable.

If $c(1,2,4,1)=0$, then $d(1,4)=0$ and
$${\scriptsize
\begin{pmatrix}
d(2,4) & d(3,4)\\
d(1,2) & d(1,3)
\end{pmatrix} =
\begin{pmatrix}
 c(1,2,4,2) c(1,3,4,4), & c(1,2,4,3) c(1,3,4,4) \\
 -c(1,2,4,2) c(1,3,4,1), & -c(1,2,4,3) c(1,3,4,1)
\end{pmatrix},}$$
has also rank less than or equal to $1$, which means that the algebra is $2$-solvable.

If $c(1,2,4,3)=0$, $c(1,3,4,4)\neq 0$ and $c(1,2,4,1)\neq 0$ then $d(1,4)\neq 0$ and
$${\scriptsize
\begin{pmatrix}
d(2,4) & d(3,4)\\
d(1,2) & d(1,3)
\end{pmatrix} =
\begin{pmatrix}
 c(1,2,4,2) c(1,3,4,4), & 0 \\
 c(1,2,4,1) c(1,3,4,2)-c(1,2,4,2) c(1,3,4,1), & c(1,2,4,1) c(1,3,4,3)
\end{pmatrix}}$$
has rank $2$ if and only if determinant $c(1,2,4,1) c(1,2,4,2) c(1,3,4,3) c(1,3,4,4)\neq 0$, that is, if and only if, $c(1,2,4,2)\neq 0$ and $c(1,3,4,3)\neq 0$.

If $c(1,2,4,4)=0$, $c(1,2,4,1)\neq 0$, $c(1,3,4,4)\neq 0$, $c(1,2,4,3)\neq 0$, $c(1,2,4,3)\neq c(1,3,4,4)$, then choosing 
\begin{align*}
& P=P_{1,2}=\\
& \begin{pmatrix}
 \frac{c(1,2,4,3)}{p(3,3)} & 0 & 0 & p(1,4) \\
p(2,1) & p(3,3) & \frac{c(1,3,4,3) p(3,3)}{c(1,2,4,3)-c(1,3,4,4)} & p(2,4) \\
 0 & 0 & p(3,3) & \frac{c(1,3,4,3) p(3,3)}{c(1,2,4,3)-c(1,3,4,4)} \\
 0 & 0 & 0 & p(3,3)
\end{pmatrix}
\end{align*}
{\scriptsize \begin{align*}
& p(1,4)= -\frac{c(1,2,4,3) \big(c(1,2,4,3) c(1,3,4,1)-c(1,3,4,4) c(1,3,4,1)-c(1,2,4,1) c(1,3,4,3)\big)}{(c(1,2,4,3)-c(1,3,4,4)) c(1,3,4,4) p(3,3)}\\
& p(2,1)= -\frac{(c(1,2,4,2) c(1,2,4,3)+c(1,3,4,3) c(1,2,4,3)-c(1,2,4,2) c(1,3,4,4)) p(3,3)}{c(1,2,4,1) (c(1,2,4,3)-c(1,3,4,4))} \\
& p(2,4)= \frac{\big(-c(1,2,4,1) c(1,3,4,3)^2+c(1,2,4,3) c(1,3,4,1) c(1,3,4,3)+c(1,2,4,2) c(1,2,4,3) c(1,3,4,1)\big) p(3,3)}{c(1,2,4,1) (c(1,2,4,3)-c(1,3,4,4)) c(1,3,4,4)} \\
&
+\frac{\big(-c(1,2,4,1) c(1,2,4,3) c(1,3,4,2)-c(1,2,4,2) c(1,3,4,1) c(1,3,4,4)+c(1,2,4,1) c(1,3,4,2) c(1,3,4,4)\big) p(3,3)}{c(1,2,4,1) (c(1,2,4,3)-c(1,3,4,4)) c(1,3,4,4)}
\end{align*}}
we get $B'=\left(
\begin{array}{cccc}
 0 & 0 & -\frac{c(1,2,4,1)}{p(3,3)^2} & 0 \\
 0 & 0 & 0 & 0 \\
 0 & 0 & -1 & 0 \\
 0 & \frac{c(1,3,4,4)}{c(1,2,4,3)} & 0 & 0
\end{array}
\right).$

If $c(1,2,4,4)=0$, $c(1,2,4,1)\neq 0$, $c(1,3,4,4)\neq 0$, $c(1,2,4,3)\neq 0$, $c(1,2,4,3)=c(1,3,4,4)$, then choosing 
\begin{align*}
& P=P_{1,3}=\\
& \left(
\begin{array}{cccc}
 \frac{c(1,3,4,4)}{p(3,3)} & 0 & 0 & \frac{c(1,2,4,1) p(2,3)-c(1,3,4,1) p(3,3)}{p(3,3)^2} \\
 \frac{-c(1,3,4,4) p(2,3)-c(1,2,4,2) p(3,3)}{c(1,2,4,1)} & p(3,3) & p(2,3) & p(2,4) \\
 0 & 0 & p(3,3) & p(2,3) \\
 0 & 0 & 0 & p(3,3)
\end{array}
\right)
\end{align*}
{\scriptsize \begin{align*}
p(2,4)= & \frac{-c(1,2,4,1) c(1,3,4,3) p(2,3)+c(1,3,4,1) c(1,3,4,4) p(2,3)}{c(1,2,4,1) c(1,3,4,4)} \\
&
+\frac{c(1,2,4,2) c(1,3,4,1) p(3,3)-c(1,2,4,1) c(1,3,4,2) p(3,3)}{c(1,2,4,1) c(1,3,4,4)}
\end{align*}}

we get $B'=\left(
\begin{array}{cccc}
 0 & 0 & -\frac{c(1,2,4,1)}{p(3,3)^2} & 0 \\
 0 & 0 & 0 & 0 \\
 0 & \frac{c(1,3,4,3)}{c(1,3,4,4)} & -1 & 0 \\
 0 & 1 & 0 & 0
\end{array}
\right).$

If $c(1,2,4,4)=0$, $c(1,2,4,1)\neq 0$, $c(1,3,4,4)\neq 0$ and $c(1,2,4,3)= 0$ then choosing 
\begin{align*}
& P=P_{1,4}=\\
& \left(
\begin{array}{cccc}
 \frac{c(1,3,4,4)}{p(3,3)} & 0 & 0 & \frac{-c(1,2,4,1) c(1,3,4,3)-c(1,3,4,1) c(1,3,4,4)}{c(1,3,4,4) p(3,3)} \\
 -\frac{c(1,2,4,2) p(3,3)}{c(1,2,4,1)} & p(3,3) & -\frac{c(1,3,4,3) p(3,3)}{c(1,3,4,4)} & p(2,4) \\
 0 & 0 & p(3,3) & -\frac{c(1,3,4,3) p(3,3)}{c(1,3,4,4)} \\
 0 & 0 & 0 & p(3,3) \\
\end{array}
\right)\end{align*}
{\scriptsize \begin{align*}
p(2,4)= -\frac{\big(-c(1,2,4,1) c(1,3,4,3)^2-c(1,2,4,2) c(1,3,4,1) c(1,3,4,4)+c(1,2,4,1) c(1,3,4,2) c(1,3,4,4)\big) p(3,3)}{c(1,2,4,1) c(1,3,4,4)^2}
\end{align*}}
we get $B'=\left(
\begin{array}{cccc}
 0 & 0 & -\frac{c(1,2,4,1)}{p(3,3)^2} & 0 \\
 0 & 0 & 0 & 0 \\
 0 & 0 & 0 & 0 \\
 0 & 1 & 0 & 0
\end{array}
\right)$
\item $\dim D^1_3(\mathcal{A})=2$, $2$-solvable of class $3$, non-nilpotent, with trivial center, which is equivalent to $d(1,4)=0$ and
$\begin{pmatrix}
d(2,4) & d(3,4)\\
d(1,2) & d(1,3)
\end{pmatrix}\neq 0$, thus $B$ takes the form  \\
$
B= \begin{pmatrix}
 0 & \lambda c(1,2,4,1) & -c(1,2,4,1) & 0 \\
 0 & c(1,3,4,2) & -c(1,2,4,2) & 0 \\
 0 & c(1,3,4,3) & -c(1,2,4,3) & 0 \\
 0 & \lambda c(1,2,4,4) & -c(1,2,4,4) & 0
\end{pmatrix}
$
if $(c(1,2,4,1),c(1,2,4,4))\neq (0,0)$, or \\
$
B= \begin{pmatrix}
 0 & c(1,3,4,1) & 0 & 0 \\
 0 & c(1,3,4,2) & -c(1,2,4,2) & 0 \\
 0 & c(1,3,4,3) & -c(1,2,4,3) & 0 \\
 0 & c(1,3,4,4) & 0 & 0
\end{pmatrix}
$
if $(c(1,2,4,1),c(1,2,4,4))= (0,0)$.

Consider first the case where $(c(1,2,4,1),c(1,2,4,4))\neq (0,0)$, then
$$B'=\left(
\begin{array}{cccc}
 0 & b'(1,2) & -\frac{c(1,2,4,1) p(1,1)+c(1,2,4,4) p(1,4)}{p(1,1) p(3,3)^2} & 0 \\
 0 & b'(2,2) & b'(2,3) & 0 \\
 0 & b'(3,2) & -\frac{c(1,2,4,4) p(2,3)+c(1,2,4,3) p(3,3)}{p(1,1) p(3,3)^2} & 0 \\
 0 & \frac{c(1,2,4,4) (\lambda p(3,3)-p(2,3))}{p(1,1) p(3,3)^2} & -\frac{c(1,2,4,4)}{p(1,1) p(3,3)} & 0
\end{array}
\right)$$
{\scriptsize \begin{align*}
b'(1,2)&=c'(1,3,4,1)=\frac{(c(1,2,4,1) p(1,1)+c(1,2,4,4) p(1,4)) (\lambda p(3,3)-p(2,3))}{p(1,1) p(3,3)^3},\\
b'(2,2)&=c'(1,3,4,2)\\
&=\frac{ p(3,3) \big(\lambda c(1,2,4,1) p(2,1)+\lambda c(1,2,4,4) p(2,4)+c(1,3,4,3) p(2,3)+c(1,3,4,2) p(3,3)\big)}{p(1,1) p(3,3)^3}\\
&\quad +\frac{ -p(2,3) \big(c(1,2,4,1) p(2,1)+c(1,2,4,3) p(2,3)+c(1,2,4,4) p(2,4)+c(1,2,4,2) p(3,3) \big) }{p(1,1) p(3,3)^3},\\
b'(3,2)&=c'(1,3,4,3)\\
&=\frac{ c(1,2,4,4) p(2,3) (\lambda p(3,3)-p(2,3)) +p(3,3) (c(1,3,4,3) p(3,3)-c(1,2,4,3) p(2,3)) }{p(1,1) p(3,3)^3},\\
b'(2,3)&=-c'(1,2,4,2)\\
&=-\frac{c(1,2,4,1) p(2,1)+c(1,2,4,3) p(2,3) +c(1,2,4,4) p(2,4)+c(1,2,4,2) p(3,3) }{p(1,1) p(3,3)^2}.
\end{align*} }

If $c(1,2,4,4)\neq 0$ then choosing 
$$P=P_{2,1}=\begin{pmatrix}
 \frac{c(1,2,4,4)}{p(3,3)} & 0 & 0 &
   -\frac{c(1,2,4,1)}{p(3,3)} \\
 p(2,1) & p(3,3) & -\frac{c(1,2,4,3) p(3,3)}{c(1,2,4,4)}
   & \frac{\scriptsize \begin{array}{c} c(1,2,4,3)^2 p(3,3)\\-c(1,2,4,1) c(1,2,4,4)
   p(2,1)\\-c(1,2,4,2) c(1,2,4,4) p(3,3)\end{array} }{c(1,2,4,4)^2} \\
 0 & 0 & p(3,3) & -\frac{c(1,2,4,3) p(3,3)}{c(1,2,4,4)}
   \\
 0 & 0 & 0 & p(3,3)
\end{pmatrix}$$
we get
$B'=\left(
\begin{array}{cccc}
 0 & 0 & 0 & 0 \\
 0 & \frac{\scriptsize \begin{array}{c} \lambda c(1,2,4,3)^2-\lambda c(1,2,4,2) c(1,2,4,4)\\ -c(1,3,4,3) c(1,2,4,3)+c(1,2,4,4) c(1,3,4,2)\end{array} }{c(1,2,4,4)^2} & 0 & 0 \\
 0 & \frac{c(1,3,4,3)-\lambda c(1,2,4,3)}{c(1,2,4,4)} & 0 & 0 \\
 0 & \frac{\lambda c(1,2,4,4)+c(1,2,4,3)}{c(1,2,4,4)} & -1 & 0
\end{array}
\right)$

If $c(1,2,4,4)=0$ and $c(1,2,4,3)\neq 0$ then 
$$B'=\left(
\begin{array}{cccc}
 0 & \frac{c(1,3,4,1) p(1,1)+c(1,3,4,4) p(1,4)}{p(1,1) p(3,3)^2} & 0 & 0 \\
 0 & \frac{\scriptsize \begin{array}{c} c(1,3,4,1) p(2,1)-c(1,2,4,2) p(2,3)+c(1,3,4,3) p(2,3)\\+c(1,3,4,4) p(2,4)+c(1,3,4,2) p(3,3)\end{array} }{p(1,1) p(3,3)^2} & -\frac{c(1,2,4,2)}{p(1,1) p(3,3)} & 0 \\
 0 & \frac{c(1,3,4,4) p(2,3)+c(1,3,4,3) p(3,3)}{p(1,1) p(3,3)^2} & 0 & 0 \\
 0 & \frac{c(1,3,4,4)}{p(1,1) p(3,3)} & 0 & 0
\end{array}
\right)$$
By choosing
$$P=P_{2,2}=\left(
\begin{array}{cccc}
 \frac{c(1,2,4,3)}{p(3,3)} & 0 & 0 & p(1,4) \\
 -\frac{(c(1,2,4,2)+c(1,3,4,3)) p(3,3)}{c(1,2,4,1)} & p(3,3) & \frac{c(1,3,4,3) p(3,3)}{c(1,2,4,3)} & p(2,4) \\
 0 & 0 & p(3,3) & \frac{c(1,3,4,3) p(3,3)}{c(1,2,4,3)} \\
 0 & 0 & 0 & p(3,3)
\end{array}
\right),$$
we get
$$B'=\left(
\begin{array}{cccc}
 0 & \frac{c(1,2,4,1) (\lambda c(1,2,4,3)-c(1,3,4,3))}{c(1,2,4,3) p(3,3)^2} & -\frac{c(1,2,4,1)}{p(3,3)^2} & 0 \\
 0 & \frac{\scriptsize \begin{array}{c}-\lambda c(1,2,4,3) c(1,3,4,3)-\lambda c(1,2,4,2) c(1,2,4,3) \\+c(1,3,4,3)^2+c(1,2,4,3) c(1,3,4,2) \end{array}}{c(1,2,4,3)^2} & 0 & 0 \\
 0 & 0 & -1 & 0 \\
 0 & 0 & 0 & 0
\end{array}
\right).$$

If $c(1,2,4,4) = 0$, in which case $c(1,2,4,1)\neq 0$ (else the algebra would be $2$-solvable of class $2$ by Theorem \ref{5cases}). We consider $c(1,2,4,3)= 0$ and $c(1,3,4,3)\neq 0$. We have
\[B'=\left(
\begin{array}{cccc}
 0 & \frac{c(1,2,4,1) (\lambda p(3,3)-p(2,3))}{p(3,3)^3} & -\frac{c(1,2,4,1)}{p(3,3)^2} & 0 \\
 0 & \frac{\scriptsize \begin{array}{c} p(3,3) \big(\lambda c(1,2,4,1) p(2,1)+c(1,3,4,3) p(2,3)\\+c(1,3,4,2) p(3,3)\big)-p(2,3)\times \\ \times (c(1,2,4,1) p(2,1)+c(1,2,4,2) p(3,3))\end{array}}{p(1,1) p(3,3)^3} & -\frac{c(1,2,4,1) p(2,1)+c(1,2,4,2) p(3,3)}{p(1,1) p(3,3)^2} & 0 \\
 0 & \frac{c(1,3,4,3)}{p(1,1) p(3,3)} & 0 & 0 \\
 0 & 0 & 0 & 0
\end{array}
\right),\]
choosing
\begin{align*}
&P=P_{2,3}\\
&=\left(
\begin{array}{cccc}
 \frac{c(1,3,4,3)}{p(3,3)} & 0 & 0 & p(1,4) \\
 -\frac{c(1,2,4,2) p(3,3)}{c(1,2,4,1)} & p(3,3) & \frac{p(3,3) (\lambda c(1,2,4,2)-c(1,3,4,2))}{c(1,3,4,3)} & p(2,4) \\
 0 & 0 & p(3,3) & \frac{p(3,3) (\lambda c(1,2,4,2)-c(1,3,4,2))}{c(1,3,4,3)} \\
 0 & 0 & 0 & p(3,3)
\end{array}
\right)
\end{align*}
we get
$B'=\left(
\begin{array}{cccc}
 0 & \frac{c(1,2,4,1) (-\lambda c(1,2,4,2)+\lambda c(1,3,4,3)+c(1,3,4,2))}{c(1,3,4,3) p(3,3)^2} & -\frac{c(1,2,4,1)}{p(3,3)^2} & 0 \\
 0 & 0 & 0 & 0 \\
 0 & 1 & 0 & 0 \\
 0 & 0 & 0 & 0
\end{array}
\right).$ \\
 Consider now $(c(1,2,4,1),c(1,2,4,4)) \neq (0,0)$ and $c(1,2,4,4) = 0$, meaning that $c(1,2,4,1)\neq 0$. 
Suppose also that $c(1,2,4,3)= 0$ and $c(1,3,4,3)=0$. Then

 \[B'=\left(
\begin{array}{cccc}
 0 & \frac{c(1,3,4,1) p(1,1) p(3,3)-c(1,2,4,1) p(1,1) p(2,3)}{p(1,1) p(3,3)^3} & -\frac{c(1,2,4,1)}{p(3,3)^2} & 0 \\
 0 & \frac{\scriptsize \begin{array}{c}p(3,3) (c(1,3,4,1) p(2,1)+c(1,3,4,2) p(3,3))\\+p(2,3) (-c(1,2,4,1) p(2,1)-c(1,2,4,2) p(3,3))\end{array} }{p(1,1) p(3,3)^3} & \frac{-c(1,2,4,1) p(2,1)-c(1,2,4,2) p(3,3)}{p(1,1) p(3,3)^2} & 0 \\
 0 & 0 & 0 & 0 \\
 0 & 0 & 0 & 0
\end{array}
\right)\]

\[P=P_{2,4}=\left(
\begin{array}{cccc}
 \frac{c(1,2,4,1) c(1,3,4,2)-c(1,2,4,2) c(1,3,4,1)}{c(1,2,4,1) p(3,3)} & 0 & 0 & p(1,4) \\
 -\frac{c(1,2,4,2) p(3,3)}{c(1,2,4,1)} & p(3,3) & \frac{c(1,3,4,1) p(3,3)}{c(1,2,4,1)} & p(2,4) \\
 0 & 0 & p(3,3) & \frac{c(1,3,4,1) p(3,3)}{c(1,2,4,1)} \\
 0 & 0 & 0 & p(3,3)
\end{array}
\right)\]

\[B'=\left(
\begin{array}{cccc}
 0 & 0 & -\frac{c(1,2,4,1)}{p(3,3)^2} & 0 \\
 0 & 1 & 0 & 0 \\
 0 & 0 & 0 & 0 \\
 0 & 0 & 0 & 0
\end{array}
\right).\]
In this case, the algebra is multiplicative by Corollary \ref{mul_sp}.

If $(c(1,2,4,1),c(1,2,4,4))= (0,0)$ and $(c(1,3,4,1),c(1,3,4,4)) \neq (0,0)$, then
\[B'=\left(
\begin{array}{cccc}
 0 & \frac{c(1,3,4,1) p(1,1)+c(1,3,4,4) p(1,4)}{p(1,1) p(3,3)^2} & 0 & 0 \\
 0 & \frac{\scriptsize \begin{array}{c}p(3,3) (c(1,3,4,1) p(2,1)-c(1,2,4,2) p(2,3)\\+c(1,3,4,3) p(2,3)+c(1,3,4,4) p(2,4)\\+c(1,3,4,2) p(3,3))-c(1,2,4,3) p(2,3)^2\end{array} }{p(1,1) p(3,3)^3} & -\frac{c(1,2,4,3) p(2,3)+c(1,2,4,2) p(3,3)}{p(1,1) p(3,3)^2} & 0 \\
 0 & \frac{\scriptsize \begin{array}{c} -c(1,2,4,3) p(2,3)+c(1,3,4,4) p(2,3)\\+c(1,3,4,3) p(3,3)\end{array} }{p(1,1) p(3,3)^2} & -\frac{c(1,2,4,3)}{p(1,1) p(3,3)} & 0 \\
 0 & \frac{c(1,3,4,4)}{p(1,1) p(3,3)} & 0 & 0
\end{array}
\right).\]

If $c(1,2,4,3)\neq 0$, choosing

\[P=P_{2,5}=\left(
\begin{array}{cccc}
 \frac{c(1,2,4,3)}{p(3,3)} & 0 & 0 & -\frac{c(1,2,4,3) c(1,3,4,1)}{c(1,3,4,4) p(3,3)} \\
 p(2,1) & p(3,3) & -\frac{c(1,2,4,2) p(3,3)}{c(1,2,4,3)} & \frac{\scriptsize\begin{array}{c}-c(1,2,4,3) c(1,3,4,1) p(2,1)\\-c(1,2,4,3) c(1,3,4,2) p(3,3)\\+c(1,2,4,2) c(1,3,4,3) p(3,3)\end{array} }{c(1,2,4,3) c(1,3,4,4)} \\
 0 & 0 & p(3,3) & -\frac{c(1,2,4,2) p(3,3)}{c(1,2,4,3)} \\
 0 & 0 & 0 & p(3,3)
\end{array}
\right),\]

we get
$B'=\left(
\begin{array}{cccc}
 0 & 0 & 0 & 0 \\
 0 & 0 & 0 & 0 \\
 0 & \frac{c(1,2,4,3) c(1,3,4,3)+c(1,2,4,2) (c(1,2,4,3)-c(1,3,4,4))}{c(1,2,4,3)^2} & -1 & 0 \\
 0 & \frac{c(1,3,4,4)}{c(1,2,4,3)} & 0 & 0
\end{array}
\right).$

If $c(1,2,4,3) = 0$, then

\[B'=\left(
\begin{array}{cccc}
 0 & \frac{c(1,3,4,1) p(1,1)+c(1,3,4,4) p(1,4)}{p(1,1) p(3,3)^2} & 0 & 0 \\
 0 & \frac{\scriptsize \begin{array}{c} c(1,3,4,1) p(2,1)-c(1,2,4,2) p(2,3)\\+c(1,3,4,3) p(2,3)+c(1,3,4,4) p(2,4)+c(1,3,4,2) p(3,3)\end{array} }{p(1,1) p(3,3)^2} & -\frac{c(1,2,4,2)}{p(1,1) p(3,3)} & 0 \\
 0 & \frac{c(1,3,4,4) p(2,3)+c(1,3,4,3) p(3,3)}{p(1,1) p(3,3)^2} & 0 & 0 \\
 0 & \frac{c(1,3,4,4)}{p(1,1) p(3,3)} & 0 & 0
\end{array}
\right).\]

If $c(1,3,4,4)\neq 0$, choosing

\[P=P_{2,6}=\left(
\begin{array}{cccc}
 \frac{c(1,3,4,4)}{p(3,3)} & 0 & 0 & -\frac{c(1,3,4,1)}{p(3,3)} \\
 p(2,1) & p(3,3) & -\frac{c(1,3,4,3) p(3,3)}{c(1,3,4,4)} & \frac{\scriptsize \begin{array}{c} c(1,3,4,3)^2 p(3,3)\\-c(1,2,4,2) c(1,3,4,3) p(3,3) \\ -c(1,3,4,1) c(1,3,4,4) p(2,1)\\-c(1,3,4,2) c(1,3,4,4) p(3,3)\end{array} }{c(1,3,4,4)^2} \\
 0 & 0 & p(3,3) & -\frac{c(1,3,4,3) p(3,3)}{c(1,3,4,4)} \\
 0 & 0 & 0 & p(3,3)
\end{array}
\right),\]
we get
$B'=\left(
\begin{array}{cccc}
 0 & 0 & 0 & 0 \\
 0 & 0 & -\frac{c(1,2,4,2)}{c(1,3,4,4)} & 0 \\
 0 & 0 & 0 & 0 \\
 0 & 1 & 0 & 0
\end{array}
\right).$

If $c(1,3,4,4)=0$, since $(c(1,3,4,1),c(1,3,4,4)) \neq (0,0)$, then $c(1,3,4,1)\neq 0$,
\[B'=\left(
\begin{array}{cccc}
 0 & \frac{c(1,3,4,1)}{p(3,3)^2} & 0 & 0 \\
 0 & \frac{\scriptsize \begin{array}{c} c(1,3,4,1) p(2,1)-c(1,2,4,2) p(2,3)\\+c(1,3,4,3) p(2,3)+c(1,3,4,2) p(3,3) \end{array} }{p(1,1) p(3,3)^2} & -\frac{c(1,2,4,2)}{p(1,1) p(3,3)} & 0 \\
 0 & \frac{c(1,3,4,3)}{p(1,1) p(3,3)} & 0 & 0 \\
 0 & 0 & 0 & 0
\end{array}
\right).\]

If $c(1,3,4,3)\neq 0$, then taking

\[P=P_{2,7}=\left(
\begin{array}{cccc}
 \frac{c(1,3,4,3)}{p(3,3)} & 0 & 0 & p(1,4) \\
 \frac{\scriptsize \begin{array}{c} c(1,2,4,2) p(2,3)-c(1,3,4,3) p(2,3)\\-c(1,3,4,2) p(3,3)\end{array} }{c(1,3,4,1)} & p(3,3) & p(2,3) & p(2,4) \\
 0 & 0 & p(3,3) & p(2,3) \\
 0 & 0 & 0 & p(3,3)
\end{array}
\right),\]
we get
$B'=\left(
\begin{array}{cccc}
 0 & \frac{c(1,3,4,1)}{p(3,3)^2} & 0 & 0 \\
 0 & 0 & -\frac{c(1,2,4,2)}{c(1,3,4,3)} & 0 \\
 0 & 1 & 0 & 0 \\
 0 & 0 & 0 & 0
\end{array}
\right).$

If $c(1,3,4,3)=0$, then the algebra is now multiplicative, and we have
\[B'=\left(
\begin{array}{cccc}
 0 & \frac{c(1,3,4,1)}{p(3,3)^2} & 0 & 0 \\
 0 & \frac{p(3,3) (c(1,3,4,1) p(2,1)+c(1,3,4,2) p(3,3))-c(1,2,4,2) p(2,3) p(3,3)}{p(1,1) p(3,3)^3} & -\frac{c(1,2,4,2)}{p(1,1) p(3,3)} & 0 \\
 0 & 0 & 0 & 0 \\
 0 & 0 & 0 & 0
\end{array}
\right).\]
As in the previous case, $c(1,3,4,1\neq 0$, moreover $c(1,2,4,2)\neq 0$ because otherwise we would have $\dim D_3^1(\mathcal{A})=1$. Choosing
\begin{align*}
& P=P_{2,8}\\
& =\left(
\begin{array}{cccc}
 \frac{c(1,2,4,2)}{p(3,3)} & 0 & 0 & p(1,4) \\
 p(2,1) & p(3,3) & \frac{c(1,3,4,1) p(2,1)+c(1,3,4,2) p(3,3)}{c(1,2,4,2)} & p(2,4) \\
 0 & 0 & p(3,3) & \frac{c(1,3,4,1) p(2,1)+c(1,3,4,2) p(3,3)}{c(1,2,4,2)} \\
 0 & 0 & 0 & p(3,3)
\end{array}
\right),
\end{align*}
we get
$B'=\left(
\begin{array}{cccc}
 0 & \frac{c(1,3,4,1)}{p(3,3)^2} & 0 & 0 \\
 0 & 0 & -1 & 0 \\
 0 & 0 & 0 & 0 \\
 0 & 0 & 0 & 0
\end{array}
\right).$

\item $\dim D^1_3(\mathcal{A})=2$, $2$-solvable of class $2$, non-nilpotent, with trivial center. In this case, the matrix defining the bracket is given by
\[
B= \begin{pmatrix}
 0 & 0 & 0 & 0 \\
 0 & c(1,3,4,2) & -c(1,2,4,2) & 0 \\
 0 & c(1,3,4,3) & -c(1,2,4,3) & 0 \\
 0 & 0 & 0 & 0
\end{pmatrix},
\]
\[B'=
\begin{pmatrix}
 0 & 0 & 0 & 0 \\
 0 & \frac{\scriptsize \begin{array}{c} p(2,3) (-c(1,2,4,3) p(2,3)-c(1,2,4,2) p(3,3))\\+p(3,3) (c(1,3,4,3) p(2,3)+c(1,3,4,2)
   p(3,3)) \end{array} }{p(1,1) p(3,3)^3} & \frac{-c(1,2,4,3) p(2,3)-c(1,2,4,2) p(3,3)}{p(1,1) p(3,3)^2} &
   0 \\
 0 & \frac{c(1,3,4,3) p(3,3)^2-c(1,2,4,3) p(2,3) p(3,3)}{p(1,1) p(3,3)^3} &
   -\frac{c(1,2,4,3)}{p(1,1) p(3,3)} & 0 \\
 0 & 0 & 0 & 0
\end{pmatrix}.
\]
Note that $c'(1,2,4,3)=0$ if and only if $c(1,2,4,3)=0$. Thus the cases where $c(1,2,4,3)=0$ and $c(1,2,4,3)\neq 0$ cannot be isomorphic.
\begin{enumerate}[wide, labelwidth=!, labelindent=0pt]
\item If $c(1,2,4,3)\neq 0$ then taking
\[ P=P_{3,1}=
\begin{pmatrix}
 \frac{c(1,2,4,3)}{p(3,3)} & 0 & 0 & p(1,4) \\
 p(2,1) & p(3,3) & \frac{c(1,3,4,3) p(3,3)}{c(1,2,4,3)} & p(2,4) \\
 0 & 0 & p(3,3) & \frac{c(1,3,4,3) p(3,3)}{c(1,2,4,3)} \\
 0 & 0 & 0 & p(3,3)
\end{pmatrix},
 \]
for arbitrary $p(2,1), p(1,4), p(2,4)$ and $p(3,3)\neq 0$, gives the following matrix defining the bracket
 \begin{align*}
& B'=\begin{pmatrix}
 0 & 0 & 0 & 0 \\
 0 & c'(1,3,4,2) & -c'(1,2,4,2) & 0 \\
 0 & 0 & -1 & 0 \\
 0 & 0 & 0 & 0
\end{pmatrix},
\\
& c'(1,3,4,2)=\frac{c(1,2,4,3) c(1,3,4,2)-c(1,2,4,2) c(1,3,4,3)}{c(1,2,4,3)^2}, \\
& c'(1,2,4,2)=\frac{-c(1,2,4,2)-c(1,3,4,3)}{c(1,2,4,3)}.
\end{align*}
\item If $c(1,2,4,3)=0$ then consider $c(1,3,4,3)\neq 0$ and $c(1,2,4,2)\neq 0$, since otherwise the center of the algebra would become non-zero (Theorem \ref{5cases})
\[
B'=\begin{pmatrix}
 0 & 0 & 0 & 0 \\
 0 & \frac{(c(1,3,4,3)-c(1,2,4,2)) p(2,3)+c(1,3,4,2) p(3,3) }{p(1,1)
   p(3,3)^2} & -\frac{c(1,2,4,2)}{p(1,1) p(3,3)} & 0 \\
 0 & \frac{c(1,3,4,3)}{p(1,1) p(3,3)} & 0 & 0 \\
 0 & 0 & 0 & 0
\end{pmatrix}.
\]
Taking \begin{align*}
& P=P_{3,2}= \\
&\begin{pmatrix}
 \frac{c(1,3,4,3)}{p(3,3)} & 0 & 0 & p(1,4) \\
 p(2,1) & p(3,3) & -\frac{(c(1,3,4,3)-c(1,3,4,2)) p(3,3)}{c(1,2,4,2)-c(1,3,4,3)} & p(2,4) \\
 0 & 0 & p(3,3) & -\frac{(c(1,3,4,3)-c(1,3,4,2)) p(3,3)}{c(1,2,4,2)-c(1,3,4,3)} \\
 0 & 0 & 0 & p(3,3)
\end{pmatrix}
\end{align*}
for arbitrary $p(2,1), p(1,4), p(2,4)$ and $p(3,3)\neq 0$ and for $c(1,3,4,3)\neq c(1,2,4,2)$, gives the following matrix defining the bracket
$B'=
\begin{pmatrix}
 0 & 0 & 0 & 0 \\
 0 & 1 & -c'(1,2,4,2) & 0 \\
 0 & 1 & 0 & 0 \\
 0 & 0 & 0 & 0
\end{pmatrix},
$
with $c'(1,2,4,2)=\frac{c(1,2,4,2)}{c(1,3,4,3)}$.

Consider now two such algebras with different parameters $c'(1,2,4,2)=a$ and $c''(1,2,4,2)=b$, and denote the matrices defining the brackets by $B'_1$ and $B'_2$ respectively. Those algebras are isomorphic if and only if \[ \frac{1}{\det(P)}P B'_1 P^T - B'_2=
\begin{pmatrix}
 0 & 0 & 0 & 0 \\
 0 & \frac{(a+1) p(2,3)}{p(1,1) p(3,3)^2} & \frac{a}{p(1,1) p(3,3)}-b & 0 \\
 0 & \frac{1}{p(1,1) p(3,3)}-1 & 0 & 0 \\
 0 & 0 & 0 & 0
\end{pmatrix}=0.\]

\item If $(c(1,3,4,3)=c(1,2,4,2)$, then
$
B'=\begin{pmatrix}
 0 & 0 & 0 & 0 \\
 0 & \frac{c(1,3,4,2) }{p(1,1)
   p(3,3)} & -\frac{c(1,2,4,2)}{p(1,1) p(3,3)} & 0 \\
 0 & \frac{c(1,2,4,2)}{p(1,1) p(3,3)} & 0 & 0 \\
 0 & 0 & 0 & 0
\end{pmatrix}.
$
\\
Taking
$P=P_{3,3}=\left(
\begin{array}{cccc}
 p(1,1) & 0 & 0 & p(1,4) \\
 p(2,1) & \frac{c(1,3,4,3)}{p(1,1)} & p(2,3) & p(2,4) \\
 0 & 0 & \frac{c(1,3,4,3)}{p(1,1)} & p(2,3) \\
 0 & 0 & 0 & \frac{c(1,3,4,3)}{p(1,1)}
\end{array}
\right)$
gives
$$B'=\left(
\begin{array}{cccc}
 0 & 0 & 0 & 0 \\
 0 & \frac{c(1,3,4,2)}{c(1,3,4,3)} & -1 & 0 \\
 0 & 1 & 0 & 0 \\
 0 & 0 & 0 & 0
\end{array}
\right).$$
\end{enumerate}

\item $\dim D^1_3(\mathcal{A})=1$, $\mathcal{A}$ is $2$-solvable of class $2$, non-nilpotent, with $1$-dimensional center.

In this case $w_2$ and $w_3$ are linearly dependent.

\noindent
If $w_3\neq 0$, $w_2=\lambda w_3$, $\lambda \in \mathbb{K}$, then
$
B= \begin{pmatrix}
 0 & \lambda c(1,2,4,1) & -c(1,2,4,1) & 0 \\
 0 & \lambda c(1,2,4,2) & -c(1,2,4,2) & 0 \\
 0 & \lambda c(1,2,4,3) & -c(1,2,4,3) & 0 \\
 0 & \lambda c(1,2,4,4) & -c(1,2,4,4) & 0
\end{pmatrix}.
$

\noindent
If $w_3=0$ and $w_2\neq 0$, then
$
B= \begin{pmatrix}
 0 & c(1,3,4,1) & 0 & 0 \\
 0 & c(1,3,4,2) & 0 & 0 \\
 0 & c(1,3,4,3) & 0 & 0 \\
 0 & c(1,3,4,4) & 0 & 0
\end{pmatrix}.
$
	\begin{enumerate}[wide, labelwidth=!, labelindent=0pt]
	\item We consider first the case when $w_3\neq 0$ and $w_2=\lambda w_3$ where $\lambda \in \mathbb{K}$, then \\
	$B'=\left(
	\begin{array}{cccc}
	 0 & \frac{\scriptsize \begin{array}{c} (\lambda p(3,3)-p(2,3)) (c(1,2,4,1) p(1,1)\\+c(1,2,4,4) p(1,4))\end{array}}{p(1,1) p(3,3)^3 } & -\frac{c(1,2,4,1) p(1,1)+c(1,2,4,4) p(1,4)}{p(1,1) p(3,3)^2} & 0 \\
	 0 & \frac{\scriptsize \begin{array}{c} (\lambda p(3,3)-p(2,3)) \big(c(1,2,4,1) p(2,1)\\+c(1,2,4,3) p(2,3)+c(1,2,4,4) p(2,4)\\+c(1,2,4,2) p(3,3)\big)\end{array} }{p(1,1) p(3,3)^3} & -\frac{\scriptsize \begin{array}{c}  c(1,2,4,1) p(2,1)+c(1,2,4,3) p(2,3)\\+c(1,2,4,4) p(2,4)+c(1,2,4,2) p(3,3) \end{array} }{p(1,1) p(3,3)^2} & 0 \\
	 0 & \frac{\scriptsize \begin{array}{c} (\lambda p(3,3)-p(2,3)) (c(1,2,4,4) p(2,3)\\+c(1,2,4,3) p(3,3))\end{array} }{p(1,1) p(3,3)^3} & -\frac{c(1,2,4,4) p(2,3)+c(1,2,4,3) p(3,3)}{p(1,1) p(3,3)^2} & 0 \\
	 0 & \frac{c(1,2,4,4) (\lambda p(3,3)-p(2,3))}{p(1,1) p(3,3)^2} & -\frac{c(1,2,4,4)}{p(1,1) p(3,3)} & 0
	\end{array}
	\right).$ \\
	If $c(1,2,4,4)\neq 0$ then taking
	\[P=P_{4,1}=\left(
	\begin{array}{cccc}
	 \frac{c(1,2,4,4)}{p(3,3)} & 0 & 0 & -\frac{c(1,2,4,1)}{p(3,3)} \\
	 p(2,1) & p(3,3) & -\frac{c(1,2,4,3) p(3,3)}{c(1,2,4,4)} & \frac{\scriptsize \begin{array}{c} c(1,2,4,3)^2 p(3,3)-\\ c(1,2,4,1) c(1,2,4,4) p(2,1)-\\ c(1,2,4,2) c(1,2,4,4) p(3,3) \end{array} }{c(1,2,4,4)^2} \\
	 0 & 0 & p(3,3) & -\frac{c(1,2,4,3) p(3,3)}{c(1,2,4,4)} \\
	 0 & 0 & 0 & p(3,3)
	\end{array}
	\right)\]
	we get
	$B'=\left(
	\begin{array}{cccc}
	 0 & 0 & 0 & 0 \\
	 0 & 0 & 0 & 0 \\
	 0 & 0 & 0 & 0 \\
	 0 & \frac{c(1,2,4,3)}{c(1,2,4,4)}+\lambda & -1 & 0
	\end{array}
	\right)$
	\item If $c(1,2,4,4)=0$, then
	\[B'=\left(
	\begin{array}{cccc}
	 0 & \frac{c(1,2,4,1) (\lambda p(3,3)-p(2,3))}{p(3,3)^3} & -\frac{c(1,2,4,1)}{p(3,3)^2} & 0 \\
	 0 & \frac{\scriptsize \begin{array}{c} (\lambda p(3,3)-p(2,3)) \big(c(1,2,4,1) p(2,1) \\ +c(1,2,4,3) p(2,3)+c(1,2,4,2) p(3,3)\big)\end{array} }{p(1,1) p(3,3)^3} & -\frac{\scriptsize \begin{array}{c} c(1,2,4,1) p(2,1)+c(1,2,4,3) p(2,3)\\+c(1,2,4,2) p(3,3)\end{array} }{p(1,1) p(3,3)^2} & 0 \\
	 0 & \frac{c(1,2,4,3) (\lambda p(3,3)-p(2,3))}{p(1,1) p(3,3)^2} & -\frac{c(1,2,4,3)}{p(1,1) p(3,3)} & 0 \\
	 0 & 0 & 0 & 0
	\end{array}
	\right).\]
	If $c(1,2,4,3)\neq 0$ and $c(1,2,4,1)\neq 0$, taking
	\[P=P_{4,2}=\left(
	\begin{array}{cccc}
	 \frac{c(1,2,4,3)}{p(3,3)} & 0 & 0 & p(1,4) \\
	 \frac{-\lambda c(1,2,4,3) p(3,3)-c(1,2,4,2) p(3,3)}{c(1,2,4,1)} & p(3,3) & \lambda p(3,3) & p(2,4) \\
	 0 & 0 & p(3,3) & \lambda p(3,3) \\
	 0 & 0 & 0 & p(3,3)
	\end{array}
	\right),\]
	we get
	$B'=\left(
	\begin{array}{cccc}
	 0 & 0 & -\frac{c(1,2,4,1)}{p(3,3)^2} & 0 \\
	 0 & 0 & 0 & 0 \\
	 0 & 0 & -1 & 0 \\
	 0 & 0 & 0 & 0
	\end{array}
	\right).$
	\item If $c(1,2,4,3)=0$ and $c(1,2,4,1)\neq 0$, then
	\[B'=\left(
	\begin{array}{cccc}
	 0 & \frac{c(1,2,4,1) (\lambda p(3,3)-p(2,3))}{p(3,3)^3} & -\frac{c(1,2,4,1)}{p(3,3)^2} & 0 \\
	 0 & \frac{\scriptsize \begin{array}{c}(\lambda p(3,3)-p(2,3))\times \\ \times (c(1,2,4,1) p(2,1)+c(1,2,4,2) p(3,3))\end{array} }{p(1,1) p(3,3)^3} & -\frac{c(1,2,4,1) p(2,1)+c(1,2,4,2) p(3,3)}{p(1,1) p(3,3)^2} & 0 \\
	 0 & 0 & 0 & 0 \\
	 0 & 0 & 0 & 0
	\end{array}
	\right),\]
	
	\[P=P_{4,3}=\left(
	\begin{array}{cccc}
	 p(1,1) & 0 & 0 & p(1,4) \\
	 -\frac{c(1,2,4,2) p(3,3)}{c(1,2,4,1)} & p(3,3) & \lambda p(3,3) & p(2,4) \\
	 0 & 0 & p(3,3) & \lambda p(3,3) \\
	 0 & 0 & 0 & p(3,3)
	\end{array}
	\right),\]
	
	\[B'=\left(
	\begin{array}{cccc}
	 0 & 0 & -\frac{c(1,2,4,1)}{p(3,3)^2} & 0 \\
	 0 & 0 & 0 & 0 \\
	 0 & 0 & 0 & 0 \\
	 0 & 0 & 0 & 0
	\end{array}
	\right).\]
	\item  If $c(1,2,4,3)\neq 0$ and $c(1,2,4,1) = 0$, then
	\[B'=\left(
	\begin{array}{cccc}
	 0 & 0 & 0 & 0 \\
	 0 & \frac{\scriptsize \begin{array}{c}(\lambda p(3,3)-p(2,3))\times \\ \times (c(1,2,4,3) p(2,3)+c(1,2,4,2) p(3,3))\end{array}}{p(1,1) p(3,3)^3} & -\frac{\scriptsize \begin{array}{c} c(1,2,4,3) p(2,3)\\ +c(1,2,4,2) p(3,3)\end{array} }{p(1,1) p(3,3)^2} & 0 \\
	 0 & \frac{c(1,2,4,3) (\lambda p(3,3)-p(2,3))}{p(1,1) p(3,3)^2} & -\frac{c(1,2,4,3)}{p(1,1) p(3,3)} & 0 \\
	 0 & 0 & 0 & 0
	\end{array}
	\right),\]
	
	\[P=P_{4,4}=\left(
	\begin{array}{cccc}
	 \frac{c(1,2,4,3)}{p(3,3)} & 0 & 0 & p(1,4) \\
	 p(2,1) & p(3,3) & -\frac{c(1,2,4,2) p(3,3)}{c(1,2,4,3)} & p(2,4) \\
	 0 & 0 & p(3,3) & -\frac{c(1,2,4,2) p(3,3)}{c(1,2,4,3)} \\
	 0 & 0 & 0 & p(3,3)
	\end{array}
	\right),\]
	
	\[B'=\left(
	\begin{array}{cccc}
	 0 & 0 & 0 & 0 \\
	 0 & 0 & 0 & 0 \\
	 0 & \frac{c(1,2,4,2)}{c(1,2,4,3)}+\lambda & -1 & 0 \\
	 0 & 0 & 0 & 0 \\
	\end{array}
	\right).\]
	
	\item  If $c(1,2,4,3)=0$ and $c(1,2,4,1) = 0$, then
	\[B'=\left(
	\begin{array}{cccc}
	 0 & 0 & 0 & 0 \\
	 0 & \frac{c(1,2,4,2) (\lambda p(3,3)-p(2,3))}{p(1,1) p(3,3)^2} & -\frac{c(1,2,4,2)}{p(1,1) p(3,3)} & 0 \\
	 0 & 0 & 0 & 0 \\
	 0 & 0 & 0 & 0
	\end{array}
	\right),\]
	
	\[P_{4,5}=\left(
	\begin{array}{cccc}
	 \frac{c(1,2,4,2)}{p(3,3)} & 0 & 0 & p(1,4) \\
	 p(2,1) & p(3,3) & \lambda p(3,3) & p(2,4) \\
	 0 & 0 & p(3,3) & \lambda p(3,3) \\
	 0 & 0 & 0 & p(3,3)
	\end{array}
	\right),\]
	
	\[B'=\left(
	\begin{array}{cccc}
	 0 & 0 & 0 & 0 \\
	 0 & 0 & -1 & 0 \\
	 0 & 0 & 0 & 0 \\
	 0 & 0 & 0 & 0
	\end{array}
	\right).\]
	
	\item Now we consider the case where $w_3=0$ and $w_2\neq 0$, we have
	\[B'=\left(
	\begin{array}{cccc}
	 0 & \frac{c(1,3,4,1) p(1,1)+c(1,3,4,4) p(1,4)}{p(1,1) p(3,3)^2} & 0 & 0 \\
	 0 & \frac{\scriptsize \begin{array}{c} c(1,3,4,1) p(2,1)+c(1,3,4,3) p(2,3) \\ +c(1,3,4,4) p(2,4)+c(1,3,4,2) p(3,3)\end{array}}{p(1,1) p(3,3)^2} & 0 & 0 \\
	 0 & \frac{c(1,3,4,4) p(2,3)+c(1,3,4,3) p(3,3)}{p(1,1) p(3,3)^2} & 0 & 0 \\
	 0 & \frac{c(1,3,4,4)}{p(1,1) p(3,3)} & 0 & 0
	\end{array}
	\right).\]
	If $c(1,3,4,4)\neq 0$, then choosing
	\[P=P_{4,6}=\left(
	\begin{array}{cccc}
	 \frac{c(1,3,4,4)}{p(3,3)} & 0 & 0 & -\frac{c(1,3,4,1)}{p(3,3)} \\
	 p(2,1) & p(3,3) & -\frac{c(1,3,4,3) p(3,3)}{c(1,3,4,4)} & \frac{\scriptsize \begin{array}{c} c(1,3,4,3)^2 p(3,3)\\ -c(1,3,4,1) c(1,3,4,4) p(2,1)\\-c(1,3,4,2) c(1,3,4,4) p(3,3)\end{array}}{c(1,3,4,4)^2} \\
	 0 & 0 & p(3,3) & -\frac{c(1,3,4,3) p(3,3)}{c(1,3,4,4)} \\
	 0 & 0 & 0 & p(3,3)
	\end{array}
	\right)\]
	we get
	$B'={\small \left(
	\begin{array}{cccc}
	 0 & 0 & 0 & 0 \\
	 0 & 0 & 0 & 0 \\
	 0 & 0 & 0 & 0 \\
	 0 & 1 & 0 & 0
	\end{array}
	\right)}.$
	
	\item If $c(1,3,4,4)=0$, then
	\[B'=\left(
	\begin{array}{cccc}
	 0 & \frac{c(1,3,4,1)}{p(3,3)^2} & 0 & 0 \\
	 0 & \frac{c(1,3,4,1) p(2,1)+c(1,3,4,3) p(2,3)+c(1,3,4,2) p(3,3)}{p(1,1) p(3,3)^2} & 0 & 0 \\
	 0 & \frac{c(1,3,4,3)}{p(1,1) p(3,3)} & 0 & 0 \\
	 0 & 0 & 0 & 0
	\end{array}
	\right).\]
	If $c(1,3,4,3)\neq 0$ then choosing
$$ P_{4,7}= {\small \left(
	\begin{array}{cccc}
	 \frac{c(1,3,4,3)}{p(3,3)} & 0 & 0 & p(1,4) \\
	 p(2,1) & p(3,3) & \frac{-c(1,3,4,1) p(2,1)-c(1,3,4,2) p(3,3)}{c(1,3,4,3)} & p(2,4) \\
	 0 & 0 & p(3,3) & \frac{-c(1,3,4,1) p(2,1)-c(1,3,4,2) p(3,3)}{c(1,3,4,3)} \\
	 0 & 0 & 0 & p(3,3)
	\end{array}
	\right), }
$$
	we get
	\[B'=\left(
	\begin{array}{cccc}
	 0 & \frac{c(1,3,4,1)}{p(3,3)^2} & 0 & 0 \\
	 0 & 0 & 0 & 0 \\
	 0 & 1 & 0 & 0 \\
	 0 & 0 & 0 & 0
	\end{array}
	\right).\]
	
%
%
	\item    If $c(1,3,4,3)=0$, then
	\[B'=\left(
	\begin{array}{cccc}
	 0 & \frac{c(1,3,4,1)}{p(3,3)^2} & 0 & 0 \\
	 0 & \frac{c(1,3,4,1) p(2,1)+c(1,3,4,2) p(3,3)}{p(1,1) p(3,3)^2} & 0 & 0 \\
	 0 & 0 & 0 & 0 \\
	 0 & 0 & 0 & 0
	\end{array}
	\right).\]
	If $c(1,3,4,1)\neq 0$, then choosing
	\[P_{4,9}=\left(
	\begin{array}{cccc}
	 p(1,1) & 0 & 0 & p(1,4) \\
	 -\frac{c(1,3,4,2) p(3,3)}{c(1,3,4,1)} & p(3,3) & p(2,3) & p(2,4) \\
	 0 & 0 & p(3,3) & p(2,3) \\
	 0 & 0 & 0 & p(3,3) \\
	\end{array}
	\right)\]
	gives
	\[B'=\left(
	\begin{array}{cccc}
	 0 & \frac{c(1,3,4,1)}{p(3,3)^2} & 0 & 0 \\
	 0 & 0 & 0 & 0 \\
	 0 & 0 & 0 & 0 \\
	 0 & 0 & 0 & 0 \\
	\end{array}
	\right).\]
	If $c(1,3,4,1)= 0$ then the algebra becomes nilpotent.
%
%
	\end{enumerate}
\item $\dim D^1_3(\mathcal{A})=1$, $2$-solvable of class $2$, nilpotent of class $2$, with $1$-dimensional center. In this case, the matrix defining the bracket of $\mathcal{A}$ takes the following form \\
$
B_{5,1}= \begin{pmatrix}
 0 & 0 & 0 & 0 \\
 0 & \frac{-c(1,2,4,2)^2}{c(1,2,4,3)} & -c(1,2,4,2) & 0 \\
 0 & -c(1,2,4,2) & -c(1,2,4,3) & 0 \\
 0 & 0 & 0 & 0 \\
\end{pmatrix},
$
where $c(1,2,4,3)\neq 0$, or \\
$
B_{5,2}= \begin{pmatrix}
 0 & 0 & 0 & 0 \\
 0 & c(1,3,4,2) & -c(1,2,4,2) & 0 \\
 0 & -c(1,2,4,2) & -\frac{-c(1,2,4,2)^2}{c(1,3,4,2)} & 0 \\
 0 & 0 & 0 & 0 \\
\end{pmatrix},
$
where $c(1,3,4,2)\neq 0$.

Consider the first form, then
\begin{align*}
& B'=\frac{1}{\det (P)} P B_{5,1} P^T\\
& ={\small \left(
\begin{array}{cccc}
 0, & 0, & 0, & 0 \\
 0, & -\frac{c(1,2,4,2)^2}{c(1,2,4,3) p(1,1) p(3,3)}-\frac{2 c(1,2,4,2) p(2,3)}{p(1,1) p(3,3)^2}-\frac{c(1,2,4,3) p(2,3)^2}{p(1,1) p(3,3)^3}, & -\frac{c(1,2,4,2)}{p(1,1) p(3,3)}-\frac{c(1,2,4,3) p(2,3)}{p(1,1) p(3,3)^2}, & 0 \\
 0, & -\frac{c(1,2,4,2)}{p(1,1) p(3,3)}-\frac{c(1,2,4,3) p(2,3)}{p(1,1) p(3,3)^2}, & -\frac{c(1,2,4,3)}{p(1,1) p(3,3)}, & 0 \\
 0, & 0, & 0, & 0 \\
\end{array}
\right)}
\end{align*}
where $c(1,2,4,3)\neq 0$. Taking
\[P=P_{5,1}=\left(
\begin{array}{cccc}
 \frac{c(1,2,4,3)}{p(3,3)} & 0 & 0 & p(1,4) \\
 p(2,1) & p(3,3) & -\frac{c(1,2,4,2) p(3,3)}{c(1,2,4,3)} & p(2,4) \\
 0 & 0 & p(3,3) & -\frac{c(1,2,4,2) p(3,3)}{c(1,2,4,3)} \\
 0 & 0 & 0 & p(3,3) \\
\end{array}
\right)\]
we get
$B'={\small \left(
\begin{array}{cccc}
 0 & 0 & 0 & 0 \\
 0 & 0 & 0 & 0 \\
 0 & 0 & -1 & 0 \\
 0 & 0 & 0 & 0 \\
\end{array}
\right)}.$

For the second form, we have

\begin{align*}
& B'=\frac{1}{\det (P)} P B_{5,2} P^T=\\
& {\small \left(
\begin{array}{cccc}
 0 & 0 & 0 & 0 \\
 0 & \frac{c(1,2,4,2)^2 p(2,3)^2}{c(1,3,4,2) p(1,1) p(3,3)^3}-\frac{2 c(1,2,4,2) p(2,3)}{p(1,1) p(3,3)^2}+\frac{c(1,3,4,2)}{p(1,1) p(3,3)} & \frac{c(1,2,4,2)^2 p(2,3)}{c(1,3,4,2) p(1,1) p(3,3)^2}-\frac{c(1,2,4,2)}{p(1,1) p(3,3)} & 0 \\
 0 & \frac{c(1,2,4,2)^2 p(2,3)}{c(1,3,4,2) p(1,1) p(3,3)^2}-\frac{c(1,2,4,2)}{p(1,1) p(3,3)} & \frac{c(1,2,4,2)^2}{c(1,3,4,2) p(1,1) p(3,3)} & 0 \\
 0 & 0 & 0 & 0 \\
\end{array}
\right),}
\end{align*}
where $c(1,3,4,2)\neq 0$. If $c(1,2,4,2)\neq 0$, then by taking
\[P=P_{5,2}=\left(
\begin{array}{cccc}
 -\frac{c(1,2,4,2)^2}{c(1,3,4,2) p(3,3)} & 0 & 0 & p(1,4) \\
 p(2,1) & p(3,3) & \frac{c(1,3,4,2) p(3,3)}{c(1,2,4,2)} & p(2,4) \\
 0 & 0 & p(3,3) & \frac{c(1,3,4,2) p(3,3)}{c(1,2,4,2)} \\
 0 & 0 & 0 & p(3,3) \\
\end{array}
\right)\]
we get
$B'={\small \left(
\begin{array}{cccc}
 0 & 0 & 0 & 0 \\
 0 & 0 & 0 & 0 \\
 0 & 0 & -1 & 0 \\
 0 & 0 & 0 & 0 \\
\end{array}
\right)}.$

If $c(1,2,4,2)=0$, then
$B'=\left(
\begin{array}{cccc}
 0 & 0 & 0 & 0 \\
 0 & \frac{c(1,3,4,2)}{p(1,1) p(3,3)} & 0 & 0 \\
 0 & 0 & 0 & 0 \\
 0 & 0 & 0 & 0 \\
\end{array}
\right),$
and by choosing
\[P=P_{5,3}=\left(
\begin{array}{cccc}
 \frac{c(1,3,4,2)}{p(3,3)} & 0 & 0 & p(1,4) \\
 p(2,1) & p(3,3) & p(2,3) & p(2,4) \\
 0 & 0 & p(3,3) & p(2,3) \\
 0 & 0 & 0 & p(3,3) \\
\end{array}
\right)\]
we get
$B'={\small \left(
\begin{array}{cccc}
 0 & 0 & 0 & 0 \\
 0 & 1 & 0 & 0 \\
 0 & 0 & 0 & 0 \\
 0 & 0 & 0 & 0 \\
\end{array}
\right)}.$
\qed \end{enumerate}
\end{proof}

\section{Examples and remarks} \label{sec:examples}
In this section, we consider some examples that show specific properties not following from results proved above, and that may lead to further investigations of the properties of $n$-Hom-Lie algebras.
The following result is a consequence of \cite[Lemma 6.2]{kms:nhominduced}.
\begin{proposition}
Let $\mathcal{A}=(A,\bracket{\cdot,\dots,\cdot},(\alpha_i)_{1\leq i \leq n-1})$ be an $n$-Hom-Lie algebra and let $I$ be an ideal of $\mathcal{A}$. Then, for all $p\in \mathbb{N}$, $2\leq k \leq n$,  $D_k^{p+1}(I)$ is a weak ideal of $D_k^{p}(I)$ and $C_k^{p+1}(I)$ is a weak ideal of $C_k^{p}(I)$. In particular, $D_k^{1}(A)$ and $C_k^{1}(A)$ are weak ideals of $\mathcal{A}$. Moreover if all the $\alpha_i, 1\leq i \leq n-1$ are Hom-algebra morphisms, then $D_k^{p+1}(I)$ is an ideal of $D_k^{p}(I)$ and $C_k^{p+1}(I)$ is an ideal of $C_k^{p}(I)$.
\end{proposition}

A consequence of this is that all the multiplicative algebras in the above classification are not simple since they have at least one non-trivial ideal ($D_3^{1}(A)$).

The element of the derived series and central descending series of $A$ for the above algebras are given by
\begin{align*}
D^1_3(\mathcal{A})&=\langle \{ c(1,2,4,1) e_1 + c(1,2,4,2) e_2 + c(1,2,4,3) e_3 + c(1,2,4,4) e_4,\\
&\quad \quad c(1,3,4,1) e_1 + c(1,3,4,2) e_2 + c(1,3,4,3) e_3 + c(1,3,4,4) e_4 \} \rangle, \\
D^1_2(\mathcal{A})&=\langle \{ c(1,2,4,1) e_1 + c(1,2,4,2) e_2 + c(1,2,4,3) e_3 + c(1,2,4,4) e_4, \\
&\quad \quad c(1,3,4,1) e_1 + c(1,3,4,2) e_2+ c(1,3,4,3) e_3 + c(1,3,4,4) e_4 \} \rangle.
\end{align*}

For the cases 1.a) and 2.a), $D_3^{1}(A)$ is not invariant under $\alpha$, that is, it is not an ideal.

\noindent Case 2.a)  In this case,
\begin{align*}
D^2_2(\mathcal{A})=\langle\{ &
(c'(1,3,4,2)c'(1,2,4,4)-c'(1,3,4,4)c'(1,2,4,2))w_3 -\\
&\quad-(c'(1,3,4,3)c'(1,2,4,4)-c'(1,3,4,4)c'(1,2,4,3))w_2,\\
& \quad-(c'(1,3,4,1)c'(1,2,4,4)-c'(1,3,4,4)c'(1,2,4,1))w_3,\\
& \quad-(c'(1,3,4,1)c'(1,2,4,4)-c'(1,3,4,4)c'(1,2,4,1))w_2,\\
& (c'(1,3,4,1)c'(1,2,4,2)-c'(1,3,4,2)c'(1,2,4,1))w_3-\\
&\quad -(c'(1,3,4,1)c'(1,2,4,3)-c'(1,3,4,3)c'(1,2,4,1))w_2 \} \rangle\\
&= \langle \{c'(1,3,4,2) w_3 - c'(1,3,4,3) w_2\}\rangle \neq \{0\}
\end{align*}
since in case 2) $\dim D_3^1(\mathcal{A}) = 2$. Denote by $v$ the generator of $D^2_2(\mathcal{A})$:
\begin{align*}
& v= c'(1,3,4,2) w_3 - c'(1,3,4,3) w_2 \\
& =-c'(1,3,4,2) e_4  - c'(1,3,4,3)(c'(1,3,4,2)e_2+c'(1,3,4,3)e_3+c'(1,3,4,4) e_4)\\
& = -c'(1,3,4,2) e_4  - c'(1,3,4,3)c'(1,3,4,2)e_2- c'(1,3,4,3)^2 e_3\\
& \quad - c'(1,3,4,3)c'(1,3,4,4) e_4\\
& =  - c'(1,3,4,3)c'(1,3,4,2)e_2- c'(1,3,4,3)^2 e_3- (c'(1,3,4,3)c'(1,3,4,4) \\
& \quad + c'(1,3,4,2)) e_4.
\end{align*}
In general, $D_2^2(\mathcal{A})$ is a weak subalgebra of $\mathcal{A}$. We study whether $D_2^2(\mathcal{A})$ can be a Hom-subalgebra in this class. To this end, we calculate the image by $\alpha$ of $D_2^2(\mathcal{A})$:
\begin{align*}
\alpha(v)&=\alpha( - c'(1,3,4,3)c'(1,3,4,2)e_2- c'(1,3,4,3)^2 e_3\\
& \quad\quad\quad - (c'(1,3,4,3)c'(1,3,4,4) + c'(1,3,4,2)) e_4)\\
&=- \alpha( c'(1,3,4,3)c'(1,3,4,2)e_2 ) - \alpha( c'(1,3,4,3)^2 e_3 ) \\
& \quad\quad\quad -\alpha( (c'(1,3,4,3)c'(1,3,4,4) + c'(1,3,4,2)) e_4 )\\
&=- c'(1,3,4,3)c'(1,3,4,2)\alpha( e_2 ) - c'(1,3,4,3)^2 \alpha(  e_3 ) \\
& \quad\quad\quad - (c'(1,3,4,3)c'(1,3,4,4) + c'(1,3,4,2)) \alpha(  e_4 )\\
&=- c'(1,3,4,3)^2   e_2 - (c'(1,3,4,3)c'(1,3,4,4) + c'(1,3,4,2)) e_3.
\end{align*}
In the case when $(c'(1,3,4,3)c'(1,3,4,4) + c'(1,3,4,2))\neq 0$, the two elements $\alpha(c'(1,3,4,2) w_3 - c'(1,3,4,3) w_2)$ and $c'(1,3,4,2) w_3 - c'(1,3,4,3) w_2$ are linearly independent, which means that $D_2^2(\mathcal{A})$ is not invariant under $\alpha$ and thus $D_2^2(\mathcal{A})$ is a weak subalgebra but not a Hom-subalgebra of $\mathcal{A}$.

If $(c'(1,3,4,3)c'(1,3,4,4) + c'(1,3,4,2))=0$, then
\begin{align*}
c'(1,3,4,2) &=-c'(1,3,4,3)c'(1,3,4,4), \\
\alpha(v)&= - c'(1,3,4,3)^2 e_2, \\
v&=  - c'(1,3,4,3)c'(1,3,4,2)e_2- c'(1,3,4,3)^2 e_3\\
&\quad -(c'(1,3,4,3)c'(1,3,4,4) + c'(1,3,4,2)) e_4\\
&= - c'(1,3,4,3)c'(1,3,4,2)e_2- c'(1,3,4,3)^2 e_3\\
&= + c'(1,3,4,3)^2 c'(1,3,4,4)e_2- c'(1,3,4,3)^2 e_3\\
&=   c'(1,3,4,3)^2 (c'(1,3,4,4)e_2- e_3)
\end{align*}
If $c'(1,3,4,4)\neq 0$, then in this case $c'(1,3,4,3)\neq 0$ because otherwise $c'(1,3,4,2)=0$ too which contradicts to the assumption $\dim D^1_3(\mathcal{A})=2$. If $c'(1,3,4,4)=0$ then $c'(1,3,4,2)=0$, and thus $c'(1,3,4,3)\neq 0$ because otherwise $\dim D_1^3(\mathcal{A})\neq 2$. Thus these elements are linearly independent since $e_2$ and $e_3$ are linearly independent.
Thus in the case 2.a), $D_2^2(\mathcal{A})$ cannot be invariant under $\alpha$ and hence $D_2^2(\mathcal{A})$ is a weak subalgebra but not a Hom-subalgebra of $\mathcal{A}$.

Since $D_2^2(\mathcal{A})$ is not a Hom-subalgebra of $\mathcal{A}$, it is not a Hom-ideal either. Let us study now whether $D^2_2(\mathcal{A})$ is a weak ideal of $\mathcal{A}$. We have
\begin{align*}
&\bracket{e_1,e_2,v}= \\
& \left[e_1,e_2,- c'(1,3,4,3)c'(1,3,4,2)e_2- c'(1,3,4,3)^2 e_3\right. \\
& \hspace{4cm} \left. - (c'(1,3,4,3)c'(1,3,4,4) + c'(1,3,4,2)) e_4\right]\\
&=- c'(1,3,4,3)c'(1,3,4,2) \bracket{e_1,e_2,e_2} - c'(1,3,4,3)^2\bracket{e_1,e_2,e_3}\\
&\hspace{4cm} - (c'(1,3,4,3)c'(1,3,4,4) + c'(1,3,4,2)) \bracket{e_1,e_2,e_4}\\
&=- c'(1,3,4,3)c'(1,3,4,2) 0 - c'(1,3,4,3)^2 0 \\
&\hspace{4cm} - (c'(1,3,4,3)c'(1,3,4,4) + c'(1,3,4,2)) e_4\\
&=- (c'(1,3,4,3)c'(1,3,4,4) + c'(1,3,4,2)) e_4.
\end{align*}

If $(c'(1,3,4,3)c'(1,3,4,4) + c'(1,3,4,2))\neq 0$, when $c'(1,3,4,3)\neq 0$, $\bracket{e_1,e_2,v}$ and $v$ are linearly independent. Thus $D_2^2(\mathcal{A})$ is a weak subalgebra, but not a weak ideal of $\mathcal{A}$.
If $c'(1,3,4,3)=0$, then $c'(1,3,4,2)\neq 0$, since $\dim D^1_3(\mathcal{A}) \neq 2$ otherwise, which contradicts the assumptions of the case 2.a). We get
\begin{align*}
& \bracket{e_1,e_3,v}=\bracket{e_1,e_3,- (c'(1,3,4,3)c'(1,3,4,4) + c'(1,3,4,2)) e_4}\\
&=- (c'(1,3,4,3)c'(1,3,4,4) + c'(1,3,4,2)) \bracket{e_1,e_3, e_4}\\
&= - (c'(1,3,4,3)c'(1,3,4,4) + c'(1,3,4,2))(c'(1,3,4,2) e_2 + c'(1,3,4,4) e_4).
\end{align*}
This element is linearly independent from $v$, and hence it is not in $D_2^2(\mathcal{A})$. Thus $D_2^2(\mathcal{A})$ is not a weak ideal of $\mathcal{A}$.
If $(c'(1,3,4,3)c'(1,3,4,4) + c'(1,3,4,2))= 0$, then $v=c'(1,3,4,3)^2 (c'(1,3,4,4)e_2- e_3)$. In this case, $\bracket{e_j,e_k,v}\neq 0$ if and only if $(j,k)=(1,4)$ or $(j,k)=(4,1)$. Therefore, we compute only $\bracket{e_1,e_4,v}$,
\begin{align*}
\bracket{e_1,e_4,v}&= \bracket{e_1,e_4, c'(1,3,4,3)^2 (c'(1,3,4,4)e_2- e_3)}\\
&=c'(1,3,4,3)^2 (\bracket{e_1,e_4,c'(1,3,4,4)e_2} -\bracket{e_1,e_4, e_3})\\
&= c'(1,3,4,3)^2 (-c'(1,3,4,4)\bracket{e_1,e_2,e_4} +\bracket{e_1,e_3, e_4})\\
&= c'(1,3,4,3)^2 \Big(-c'(1,3,4,4)e_4 + c'(1,3,4,2)e_2 \\
&\hspace{4cm} +c'(1,3,4,3)e_3+c'(1,3,4,4)e_4\Big)\\
&= -c'(1,3,4,3)^2 (-c'(1,3,4,3)c'(1,3,4,4)e_2 + c'(1,3,4,3)e_3)\\
&= c'(1,3,4,3)^3 (c'(1,3,4,4)e_2 -e_3)\\
&= c'(1,3,4,3) v.
\end{align*}
Therefore $D_2^2(\mathcal{A})$ is a weak ideal of $\mathcal{A}$. In this case the bracket of $\mathcal{A}$ is given by
			\begin{align*}
			\bracket{e_1,e_2,e_3}&=0\\
			\bracket{e_1,e_2,e_4}&= e_4 \\
			\bracket{e_1,e_3,e_4}&= -c'(1,3,4,3)c'(1,3,4,4) e_2 + c'(1,3,4,3) e_3 + c'(1,3,4,4) e_4\\
			\bracket{e_2,e_3,e_4}&=0,
			\end{align*}
where $c'(1,3,4,3)\neq 0$ 

%
\begin{example}
If we take $\mathbb{K}=\mathbb{C}$, $c(1,3,4,4)=\pm i$ and $c(1,3,4,3)=-2$ then we get the following two examples where $D_2^2(\mathcal{A})$ is a weak ideal of $\mathcal{A}$:
		$$	\begin{array}{rl}
			\bracket{e_1,e_2,e_3}&=0\\
			\bracket{e_1,e_2,e_4}&= e_4 \\
			\bracket{e_1,e_3,e_4}&= 2i e_2 -2 e_3 + i e_4  \\
			\bracket{e_2,e_3,e_4}&=0
			\end{array}
\quad \quad \text{or} \quad \quad
 			\begin{array}{rl}
			\bracket{e_1,e_2,e_3}&=0\\
			\bracket{e_1,e_2,e_4}&= e_4 \\
			\bracket{e_1,e_3,e_4}&= -2i e_2 - 2 e_3 - i e_4  \\
			\bracket{e_2,e_3,e_4}&=0.
			\end{array}  $$
\end{example}


\end{document}